\documentclass[12pt]{amsart}
\pdfoutput=1
\usepackage{amssymb}
\usepackage[noadjust]{cite}
\usepackage{array}
\usepackage{booktabs}
\usepackage{mdwtab}
\usepackage{mathtools}
\usepackage{lmodern}
\usepackage{microtype}
\usepackage[T1]{fontenc}
\usepackage[utf8]{inputenc}
\usepackage{hyphenat}
\usepackage{enumitem}
\usepackage{mathabx}
\usepackage{color}
\usepackage{caption}
\usepackage[labelfont=normalfont]{subcaption}
\usepackage{minibox} 
\usepackage[pdftex]{graphicx}
\usepackage[pdftex,margin=1in,rmargin=1in,marginparwidth=0.9in]{geometry}
\usepackage{xr-hyper}
\usepackage[bookmarks=true, bookmarksopen=true,%
    bookmarksdepth=3,bookmarksopenlevel=2,%
    colorlinks=true,%
    linkcolor=blue,%
    citecolor=blue,%
    filecolor=blue,%
    menucolor=blue,%
    urlcolor=blue]{hyperref}
\hypersetup{pdftitle={A positive characterization of rational maps}}
\hypersetup{pdfauthor={Dylan P. Thurston}}
\usepackage{url}
\usepackage{tikz}
\usetikzlibrary{arrows,calc,shapes.geometric,shapes.symbols,decorations.pathreplacing}
\tikzstyle{every picture}=[> = to]
\tikzset{cdlabel/.style={execute at begin node=$\scriptstyle,execute at end node=$}}
\tikzset{implication/.style={double equal sign distance, -implies}}
\tikzset{biimplication/.style={double equal sign distance, implies-implies}}

\makeatletter
\newcommand\mi@kern[1]{%
  \settowidth\@tempdima{$\mi@obj^{#1}$}
  \kern-\@tempdima
  #1
  \settowidth\@tempdima{$\mi@obj$}
  \kern\@tempdima
}

\newtoks\mi@toksp
\newtoks\mi@toksb
\DeclareRobustCommand{\manyindices}[5]{
  \def\mi@obj{#5}
  \mi@toksp\expandafter{\mi@kern{#2}}
  \mi@toksb\expandafter{\mi@kern{#1}}
  \@mathmeasure4\textstyle{#5_{#1}^{#2}}
  \@mathmeasure6\textstyle{#5_{#3}^{#4}}
  \dimen0-\wd6 \advance\dimen0\wd4
  \@mathmeasure8\textstyle{\hphantom{{}_{#1}^{#2}}#5^{\the\mi@toksp#4}_{\the\mi@toksb#3}}
  \hbox to \dimen0{}{\kern-\dimen0\box8}
}
\makeatother 

\newread\testin

\def\mathcenter#1{\vcenter{\hbox{$#1$}}}
\def\grapha#1{\includegraphics{#1}}
\def\graphb#1{\includegraphics[trim=-1 -1 -1 -1]{#1}}
\def\mfig#1{\mathcenter{\grapha{#1}}}
\def\mfigb#1{\mathcenter{\graphb{#1}}}

\renewcommand{\colon}{\nobreak\mskip2mu\mathpunct{}\nonscript
  \mkern-\thinmuskip{:}\allowbreak\mskip6muplus1mu\relax}

\ifpdf
  
\else
  
\fi

\newcommand{\RR}{\mathbb R}
\newcommand{\CC}{\mathbb C}
\newcommand{\DD}{\mathbb D}

\newcommand{\PP}{\mathbb P}

\newcommand{\co}{\colon}
\newcommand{\eps}{\varepsilon}
\renewcommand{\epsilon}{\varepsilon}
\newcommand{\abs}[1]{\lvert #1 \rvert}
\newcommand{\norm}[1]{\lVert #1 \rVert}

\renewcommand{\phi}{\varphi} 

\newcommand{\bdy}{\partial}

\DeclareMathOperator*{\esssup}{ess\,sup}

\theoremstyle{plain}
\numberwithin{equation}{section}
\newtheorem{proposition}[equation]{Proposition}
\newtheorem{lemma}[equation]{Lemma}
\newtheorem{corollary}[equation]{Corollary}
\newtheorem{conjecture}[equation]{Conjecture}

\newtheorem{theorem}{Theorem}
\newtheorem{citethm}[equation]{Theorem}

\theoremstyle{definition}
\newtheorem{definition}[equation]{Definition}

\newtheorem{question}[equation]{Question}

\theoremstyle{remark}
\newtheorem{example}[equation]{Example}
\newtheorem{remark}[equation]{Remark}

\theoremstyle{plain}

\newtheorem{taggedthmx}{Theorem}
\newenvironment{taggedthm}[1]
 {\renewcommand\thetaggedthmx{#1}\taggedthmx}
 {\endtaggedthmx}

\DeclareMathOperator{\EL}{EL} 
\DeclareMathOperator{\SF}{SF} 
\newcommand{\SFgen}{\SF_{\mathrm{gen}}} 
\newcommand{\SFsimp}{\SF_{\mathrm{simp}}} 

\DeclareMathOperator{\Edge}{Edge}
\newcommand{\Edges}{\Edge}
\DeclareMathOperator{\Lip}{Lip} 
\DeclareMathOperator{\Emb}{Emb} 
\DeclareMathOperator{\Teich}{Teich}
\DeclareMathOperator{\Area}{Area}
\renewcommand{\Join}{\mathop{\mathrm{Join}}\nolimits} 
\newcommand{\ASF}{\overline{\SF}} 
\newcommand{\ALip}{\overline{\Lip}}
\newcommand{\CCa}{\widehat{\CC}}

\newcommand{\id}{\mathrm{id}}

\newcommand{\shortseq}[5]{#1 \overset{#2}{\longrightarrow} #3 \overset{#4}{\longrightarrow} #5}

\newcommand{\wt}[1]{\widetilde{#1}}
\newcommand{\wh}[1]{\widehat{#1}}

\newcommand{\oE}{\overline{E}{}}

\newcommand{\pdual}{p^\vee}

\definecolor{dark-green}{rgb}{0,0.6,0}
\definecolor{dark-red}{rgb}{0.7,0,0}
\definecolor{dark-blue}{rgb}{0,0,0.8}

\newcommand{\dinkusimg}[2]{%
  \raisebox{-0.5\height}{\includegraphics[height=#2]{#1}}%
}
\newcommand{\dinkus}{%
  \centerline{%
    \dinkusimg{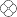}{8pt}%
    \hspace{1em}%
    \dinkusimg{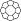}{10pt}%
    \hspace{1em}%
    \dinkusimg{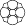}{12pt}%
    \hspace{1em}%
    \dinkusimg{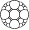}{14pt}%
    \hspace{1em}%
    \dinkusimg{dinkus-4}{12pt}%
    \hspace{1em}%
    \dinkusimg{dinkus-3}{10pt}%
    \hspace{1em}%
    \dinkusimg{dinkus-2}{8pt}%
  }
  \smallskip%
}

\graphicspath{{draws/}{figs/}}

\makeatletter
\@namedef{subjclassname@2020}{%
  \textup{2020} Mathematics Subject Classification}
\makeatother

\begin{document}
\title{A positive characterization of rational maps}

\author[Dylan Thurston]{Dylan~P.~Thurston}
\address{Department of Mathematics,
         Indiana University,
         Bloomington, Indiana 47405,
         USA}
\email{dpthurst@indiana.edu}
\date{May 4, 2020}

\begin{abstract}
  When is a topological branched self-cover of
  the sphere equivalent
  to a post-critically finite rational map on $\CC\PP^1$? William
  Thurston gave one answer in
  1982, giving a negative criterion (an obstruction to a map being
  rational). We give a complementary, positive criterion:
  the branched self-cover is equivalent to a rational map if and only
  if there is an elastic graph spine for the complement of the
  post-critical set that gets ``looser'' under
  backwards iteration.
\end{abstract}

\dedicatory{This paper is dedicated to William Thurston, 1946--2012, and Tan Lei,
  1963--2016.\\
  I will miss the joy they brought to the subject.}


\maketitle

\setcounter{tocdepth}{1}
\tableofcontents

\section{Introduction}
\label{sec:intro}

In this paper, we complete the program laid out in earlier work
\cite{Thurston16:RubberBands}, and give a positive characterization of
post-critically finite rational maps among branched self-covers of
the sphere.

\begin{definition}
  A \emph{topological branched self-cover} of the sphere is a finite
  set of points~$P \subset S^2$ and a map $f \co (S^2,P) \to (S^2,P)$,
  also written $f \co (S^2,P)\righttoleftarrow$, so that $f$ is an orientation-preserving
  covering map (with degree greater than~$1$) when restricted to a map
  from $S^2 \setminus f^{-1}(P)$ to $S^2 \setminus P$. That is, $f$~is
  a branched cover so that $f(P) \subset P$ and $P$ contains the
  critical
  values. (As a result, $P$ contains the post-critical set of~$f$.)
  Two branched
  self-covers are \emph{equivalent} if they are related by conjugacy
  of~$S^2$ (possibly changing the set~$P$) and homotopy
  relative to~$P$.
\end{definition}

One source of topological branched self-covers is post-critically
finite rational maps. Let $\CCa = \CC\PP^1$, and
suppose $f(z) = P(z)/Q(z)$ is a rational map with a finite,
forward-invariant set~$P$ that contains all critical values. Then,
if we forget the conformal structure,
$f \co (\CCa, P)\righttoleftarrow$ is a topological branched
self-cover.

\begin{question}\label{quest:pcf}
  When is a topological branched self-cover equivalent to a
  post-critically finite rational map?
\end{question}

One answer to Question~\ref{quest:pcf} was given by W.\ Thurston 30 years ago
\cite{DH93:ThurstonChar}, recalled as
Theorem~\ref{thm:thurston-obstruction}.
He proved a negative characterization: there
is a certain combinatorial object (an \emph{annular obstruction}) that
exists exactly when
$f \co (S^2,P) \righttoleftarrow$ is \emph{not} equivalent to a
rational map.
In this paper, we give a complementary,
positive, characterization: a combinatorial
object that exists exactly when $f$ \emph{is} equivalent to a
rational map.

Before stating the main theorem, we give a combinatorial description of
topological branched self-covers $f \co (S^2,P)\righttoleftarrow$ in
terms of graph maps.%
\footnote{In this paper, a graph map is a continuous map, not necessarily
  taking vertices to vertices.}
Pick a graph spine~$\Gamma_0$ for $S^2 \setminus P$ (a deformation
retract of $S^2\setminus P$, i.e., an embedded graph so each
complementary region is a punctured disk) and consider its inverse
image
$\Gamma_1 = f^{-1}(\Gamma_0) \subset S^2 \setminus f^{-1}(P)$.
There
are two natural homotopy classes of maps from $\Gamma_1$ to~$\Gamma_0$.
\begin{itemize}
\item A covering map~$\pi_\Gamma$ commuting with the action of~$f$.
  \[
  \begin{tikzpicture}
    \matrix[row sep=0.8cm,column sep=1cm] {
      \node (Gammai) {$\Gamma_1$}; &
        \node (Gamma) {$\Gamma_0$}; \\
      \node (S2i) {$S^2\setminus f^{-1}(P)$}; &
        \node (S2) {$S^2\setminus P$.}; \\
    };
    \draw[right hook->] (Gamma) to (S2);
    \draw[->] (S2i) to node[auto=left,cdlabel] {f} (S2);
    \draw[->] (Gammai) to node[auto=left,cdlabel] {\pi_\Gamma} (Gamma);
    \draw[right hook->] (Gammai) to (S2i);
  \end{tikzpicture}
  \]
\item A map~$\phi_\Gamma$ commuting up to homotopy with
  the inclusion of $S^2 \setminus f^{-1}(P)$ in $S^2 \setminus P$.
  The homotopy class~$[\phi_\Gamma]$ is unique, since
  $\Gamma_0$ is a deformation retract of $S^2\setminus P$.
  \[
  \begin{tikzpicture}
    \matrix[row sep=0.8cm,column sep=1cm] {
      \node (Gammai) {$\Gamma_1$}; &
        \node (Gamma) {$\Gamma_0$}; \\
      \node (S2i) {$S^2\setminus f^{-1}(P)$}; &
        \node (S2) {$S^2\setminus P$.}; \\
    };
    \draw[right hook->] (Gamma) to (S2);
    \draw[right hook->] (S2i) to node[auto=left,cdlabel] {\text{incl.}} (S2);
    \draw[->] (Gammai) to node[auto=left,cdlabel] {\phi_\Gamma} (Gamma);
    \draw[right hook->] (Gammai) to (S2i);
    \node at ($(Gammai)!0.5!(S2)$) {$\sim$};
  \end{tikzpicture}
  \]
\end{itemize}
This data $\pi_\Gamma, \phi_\Gamma \co \Gamma_1 \rightrightarrows \Gamma_0$ is a
\emph{virtual endomorphism} of~$\Gamma_0$. It,
together with a ribbon graph structure on~$\Gamma_0$, determines $f$ up
to equivalence. (See Section~\ref{sec:spines} for examples. This is
proved in Theorem~\ref{thm:spine-surf} in
Section~\ref{sec:spines}. Example~\ref{examp:rabbits} shows that the
extra ribbon structure is necessary to determine~$f$.)

For our characterization of rational maps, we also need an
\emph{elastic structure} on~$\Gamma = \Gamma_0$, by which we mean
an elastic length~$\alpha(e) \in \RR_{>0}$ on each edge~$e$ of~$\Gamma$. (We
treat $\alpha$ as an ordinary length for purposes of differentiation,
etc.) An
\emph{elastic graph}~$G = (\Gamma, \alpha)$ is a graph~$\Gamma$ and
an elastic structure~$\alpha$ on~$\Gamma$.
For a piecewise-linear (PL) map $\psi\co G_1 \to G_2$ between elastic
graphs, the \emph{embedding energy} is
\begin{equation}
    \label{eq:embedding-1}
    \Emb(\psi) \coloneqq
      \esssup_{y \in G_2} \sum_{x \in \psi^{-1}(y)} \abs{\psi'(x)}.
  \end{equation}
(We identify each edge~$e$ with an interval of length
$\alpha(e)$ to compute derivatives.)
The essential supremum ignores sets
of measure zero, which for PL maps amounts to ignoring
vertices of~$G_2$ and images of vertices
of~$G_1$.
On a homotopy class,
$\Emb[\psi]$ is defined to be the infimum of
$\Emb(\phi)$ for $\phi \in [\psi]$.
$\Emb[\psi]$ is realized and
controls whether $G_1$ is ``looser'' than~$G_2$ as an elastic graph
\cite[Theorem~\ref*{Elast:thm:emb-sf}]{Thurston19:Elastic}.

In the context of a branched self-cover $f \co
(S^2,P)\righttoleftarrow$, if $G = G_0$ is an spine for
$S^2\setminus P$ with an elastic structure, we get a virtual endomorphism
$\pi_G,\phi_G \co G_1 \rightrightarrows G_0$, where $G_1$ inherits an
elastic structure by
pulling back via~$\pi_G$. We can
then consider $\Emb[\phi_G]$, the embedding energy
of~$\phi_G \co G_1 \to G_0$, or the
iterated version $\Emb[\phi_G^n]$. (See Section~\ref{sec:iter} for iteration.)

In a mild generalization, we also consider disconnected surfaces.

\begin{definition}
A \emph{branched self-cover}
$f \co (\Sigma,P) \righttoleftarrow$ is a (possibly disconnected) compact
closed surface~$\Sigma$, a finite subset~$P \subset \Sigma$,
and a map $f \co \Sigma\to\Sigma$ that
\begin{itemize}
\item is a branched covering map with branch values contained in~$P$,
\item has constant degree greater than~$1$,
\item maps $P$ to~$P$, and
\item is a bijection on components of~$\Sigma$.
\end{itemize}
\end{definition}
(The restriction to $\pi_0$-bijective maps avoids dynamically
uninteresting cases.)
A standard Euler characteristic argument shows that each
component of~$\Sigma$ is either a sphere or a torus, and that in the torus
case there is no branching.

\begin{definition}\label{def:riemann-surf}
  For the purposes of this paper, a \emph{Riemann surface}
  $S = (\Sigma, \omega)$ is a
  topological surface~$\Sigma$, possibly disconnected or with
  boundary, and a conformal structure~$\omega$
  on~$\Sigma$. A \emph{rational map} is a closed Riemann surface~$S$
  and a conformal, $\pi_0$-bijective map $f \co S \righttoleftarrow$.
\end{definition}

\begin{definition}
  A
  branched self-cover $f \co (\Sigma,P) \righttoleftarrow$
  is of \emph{non-compact type} if each component of~$\Sigma$ contains a
  point of~$P$ that eventually falls into a cycle with a branch point
  (under forward iteration of~$f$). It is of \emph{hyperbolic type} if
  each component of~$\Sigma$ contains a point of~$P$ and each cycle of~$P$
  contains a branch point.
\end{definition}

In either case, the branching is non-trivial, so
$\Sigma$ is a union of spheres.  If $f$ is a rational map, it is of
non-compact type iff
the Julia set does not contain any component of~$\Sigma$ and it is
of hyperbolic type iff
the dynamics on the Julia set is hyperbolic.
Thus the term ``non-compact type'' refers to the orbifold of~$f$,
while the term ``hyperbolic type'' refers to dynamics of~$f$ and not to
the orbifold.

\begin{theorem}\label{thm:detect-rational}
  Let $f \co (\Sigma,P) \righttoleftarrow$ be a branched
  self-cover of hyperbolic type. Then the following conditions are
  equivalent.
  \begin{enumerate}
  \item The branched self-cover $f$ is equivalent to a rational map.
  \item There is an elastic graph spine~$G$ for
    $\Sigma\setminus P$ and an integer $n > 0$ so that $\Emb[\phi_G^n] <
    1$.
  \item For every elastic graph spine~$G$ for
    $\Sigma\setminus P$ and for every sufficiently large~$n$
    (depending on $f$ and~$G$), we have $\Emb[\phi_G^n] < 1$.
  \end{enumerate}
\end{theorem}
Loosely speaking, Theorem~\ref{thm:detect-rational} says that $f$ is
equivalent to a rational map iff elastic graph spines get looser under
repeated backwards iteration. As compared to the earlier
Theorem~\ref{thm:thurston-obstruction},
Theorem~\ref{thm:detect-rational} makes it easier to prove a map is
rational: you can just exhibit an elastic graph spine~$G$ and a suitable map
in the homotopy class of~$\phi_G^n$. (In practice, $n=1$ often
suffices; see Section~\ref{sec:future}.)
We prove Theorem~\ref{thm:detect-rational} as
Theorem~\ref{thm:detect-rational-a} in Section~\ref{sec:iter},
including some additional equivalent conditions.

Theorem~\ref{thm:detect-rational} and the older
Theorem~\ref{thm:thurston-obstruction} are complementary; neither one
implies the other, and the proofs are largely independent. It is easy to see
one implication: if there is an elastic graph spine that gets looser
under backwards iteration in the sense of
Theorem~\ref{thm:detect-rational}, then there is no annular
obstruction in the sense of Theorem~\ref{thm:thurston-obstruction}.
See \cite[Section~8.4]{Thurston16:RubberBands} for the argument, or
Section~\ref{sec:obstructions} of this paper for generalizations.

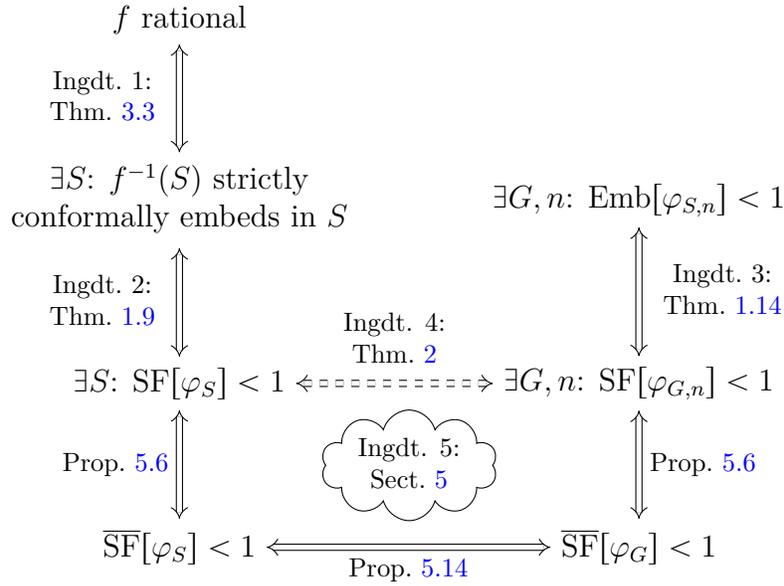
\begin{figure}
\begin{equation*}
\begin{tikzpicture}
  \matrix[row sep=8ex, column sep=4em] {
    \node (rational) {$f$ rational}; \\
    \node (emb-surf) {\minibox[c]{%
        $\exists S$: $f^{-1}(S)$ strictly\\
        conformally embeds in $S$}}; &
      \node (emb-graph) {$\exists G, n$: $\Emb[\phi_S^n] < 1$}; \\
    \node (SF-surf) {$\exists S$: $\SF[\phi_S] < 1$}; &
      \node (SF-graph) {$\exists G, n$: $\SF[\phi_G^n] < 1$}; \\
    \node (ASF-surf) {$\ASF[\pi_S,\phi_S] < 1$}; &
      \node (ASF-graph) {$\ASF[\pi_G,\phi_G] < 1$}; \\
  };
  \node[shape=cloud,cloud ignores aspect,draw,inner sep=-2pt] at ($(ASF-surf)!0.5!(SF-graph)$) {
    \footnotesize{\begin{tabular}{c}
                    Ingdt. 5:\\
                    Sect.\ \ref{sec:iter}
                  \end{tabular}}};
  \draw[biimplication] (rational) to
    node[left]{\footnotesize{\begin{tabular}{c}
      Ingdt.\ 1:\\
      Thm.\ \ref{thm:rational-surfaces-embed}
    \end{tabular}}} (emb-surf);
  \draw[biimplication] (emb-surf) to
    node[left,align=center]{\footnotesize{\begin{tabular}{c}
      Ingdt.\ 2:\\
      Thm.\ \ref{thm:emb-surf}
    \end{tabular}}} (SF-surf);
  \draw[biimplication] (SF-surf) to
    node[left]{\footnotesize{Prop.\ \ref{prop:asympt-e}}}
    (ASF-surf);
  \draw[biimplication] (ASF-surf) to
    node[below,align=center]{\footnotesize{Prop.\ \ref{prop:asf-graph-surface}}} (ASF-graph);
  \draw[biimplication] (SF-graph) to
    node[right]{\footnotesize{Prop.\ \ref{prop:asympt-e}}}
    (ASF-graph);
  \draw[biimplication] (emb-graph) to
    node[right,align=center]{\footnotesize{\begin{tabular}{c}
      Ingdt.\ 3:\\
      Thm.\ \ref{thm:emb-sf}
    \end{tabular}}}
    (SF-graph);
  \draw[biimplication,dashed] (SF-graph) to 
  node[above]{\footnotesize{\begin{tabular}{c}
                              Ingdt. 4:\\
                              Thm.\ \ref{thm:embedding-thickening}
    \end{tabular}}} (SF-surf);
\end{tikzpicture}
\end{equation*}
  \caption{An outline of the equivalences used in proving
    Theorem~\ref{thm:detect-rational}, for a fixed branched self-cover
  $f \co (\Sigma, P) \righttoleftarrow$}\label{fig:outline}
\end{figure}

There are several ingredients to prove
Theorem~\ref{thm:detect-rational}, as outlined in
Figure~\ref{fig:outline} and summarized below. Much of this has
appeared in other papers; the main new contributions of this paper are
in Ingredients~4 and~5, in Sections~\ref{sec:thicken}
and~\ref{sec:iter}, respectively. In the proof as a
whole, the hardest parts are Ingredient~1, which has been known for
some time; Ingredient~2 relating conformal embeddings to extremal
length, particularly the behavior under covers; and
Ingredient~3, which involves a combinatorial understanding
of the embedding energy in Equation~\eqref{eq:embedding-1}.

\dinkus
The
zeroth ingredient is the graphical
description of topological branched self-covers in terms of spines,
crucial to our entire approach. This is essentially a graphical
version of
Nekrashevych's automata for iterated monodromy groups
\cite{Nekrashevych05:SelfSimilar}. It is
described in Section~\ref{sec:spines}, culminating in
Theorem~\ref{thm:spine-surf}, giving a graphical model for branched self-covers.

\dinkus

The first ingredient is a characterization of rational maps in terms
of conformal embeddings of Riemann surfaces, a surface version of the
graph criterion in Theorem~\ref{thm:detect-rational}. This has been
folklore in the
community for some time and is recalled as
Theorem~\ref{thm:rational-surfaces-embed} in
Section~\ref{sec:quasi-conf-surg}.

\dinkus

The second ingredient is a characterization of conformal embeddings of
Riemann
surfaces in terms of extremal length of multi-curves on the
surface. This appeared in earlier work with Kahn and Pilgrim
\cite{KPT15:EmbeddingEL}, as we
now summarize. Recall that the \emph{extremal length} of a
simple multi-curve~$c$ on a Riemann surface measures the maximum thickness of
a collection of annuli around~$c$.

\begin{definition}\label{def:sf-surf}
  For
  $f \co R \hookrightarrow S$ a topological embedding of Riemann
  surfaces, the (extremal length) \emph{stretch factor} of~$f$ is the
  maximal ratio of extremal lengths between the two surfaces:
  \begin{equation*}
    \SF[f] \coloneqq \sup_{c\co C \to R}
      \frac{\EL_S[f\circ c]}{\EL_R[c]},
  \end{equation*}
  where the supremum runs over all simple multi-curves~$c$ on~$R$
  with $\EL_R[c] \ne 0$.
\end{definition}

\begin{definition}
  An \emph{annular extension} of a
  Riemann surface~$R$ is any surface obtained by attaching a conformal
  annulus to each boundary component of~$R$, and
  a conformal embedding $f \co R \hookrightarrow S$
  between Riemann surfaces is \emph{annular} if it extends
  to a conformal embedding of an annular extension of~$R$ into~$S$.
\end{definition}

\begin{citethm}[Kahn-Pilgrim-Thurston \cite{KPT15:EmbeddingEL}]
  \label{thm:emb-surf}
  Let $R$ and $S$ be Riemann surfaces and let $f\co
  R \hookrightarrow S$ be a
  topological embedding so that no component of $f(R)$ is contained in
  a disk or a once-punctured disk. Then $f$
  is homotopic to a conformal embedding if and only if $\SF[f] \le 1$.
  Furthermore, $f$ is homotopic to an annular conformal embedding if
  and only if $\SF[f] < 1$.
\end{citethm}

We also use Theorem~\ref{thm:sf-cover}, a strengthening of
Theorem~\ref{thm:emb-surf} that behaves well under covers.

\dinkus

The third ingredient is a relation between the embedding energy of
Equation~\eqref{eq:embedding-1} to a stretch factor of maps between
graphs (rather than surfaces)
\cite{Thurston19:Elastic}.

\begin{definition}\label{def:sf-graph}
  Let $G = (\Gamma, \alpha)$ be an elastic graph.  A \emph{multi-curve}
  on~$G$ is a (not necessarily connected) 1-manifold~$C$ and a PL map
  $c \co C \to G$. It is (strictly)
  \emph{reduced} if $c$ is locally injective. The \emph{extremal length}
  of $(C,c)$ is
  \begin{equation}\label{eq:EL-int}
  \EL(c) \coloneqq \int_{y \in G} n_c(y)^2\,d\alpha(y)
  \end{equation}
  where $n_c(y)$ is the number of elements in $c^{-1}(y)$.
  (See \cite[Section~5.2]{Thurston16:RubberBands} for motivation on why
  this is called extremal length.)
  If $c$ is reduced, then
  $n_c(y)$ depends only on the edge containing~$y$, and
  Equation~\eqref{eq:EL-int} reduces
  to
  \begin{equation}\label{eq:EL-sum}
  \EL(c) = \sum_{e \in \Edges(G)} \alpha(e) n_c(e)^2.
  \end{equation}
  $\EL[c]$ is the
  extremal length of any reduced representative of~$[c]$.

  For $[\phi] \co G \to H$ a homotopy class of maps
  between elastic graphs, the (extremal length) \emph{stretch factor}
  is the maximum ratio of extremal lengths:
  \begin{equation}\label{eq:sf-el}
    \SF[\phi] \coloneqq \sup_{c\co C \to G} \frac{\EL_H[\phi \circ c]}{\EL_G[c]}
  \end{equation}
  where the supremum runs over all non-trivial multi-curves on~$G$.
\end{definition}

\begin{citethm}[{\cite[Theorem~\ref*{Elast:thm:emb-sf}]{Thurston19:Elastic}}]
  \label{thm:emb-sf}
  For $[\phi] \co G \to H$ a homotopy class of maps
  between elastic graphs,
  \begin{equation*}
  \Emb[\phi] = \SF[\phi].
  \end{equation*}
\end{citethm}

The above two quantities are also equal to the maximum ratio of
Dirichlet energies between the two graphs. This arises
naturally in
the proof of Theorem~\ref{thm:emb-sf},
and justifies the terminology of ``loosening'' of elastic
graphs. This
fact is not used in the present paper, so we
will not develop it further here.

\dinkus

The fourth ingredient is a relation between extremal lengths on a graph
and on a certain degenerating family of surfaces.
Suppose that $G$ is an elastic ribbon
graph, where a \emph{ribbon graph} has a specified counterclockwise
cyclic order on edges
incident to each vertex (Definition~\ref{def:ribbon-graph}).
Its \emph{$\eps$-thickening} $N_\eps G$ is the
conformal surface obtained by replacing each edge~$e$ of~$G$ by a
rectangle of size $\alpha(e) \times \eps$ and gluing the rectangles at
the vertices using the given cyclic order, as shown in
Figure~\ref{fig:geom-thicken}.
A \emph{ribbon map} $\phi \co G_1 \to G_2$ between ribbon graphs is a
map that lifts to a topological embedding
$N_\eps\phi \co N_\eps G_1 \hookrightarrow N_\eps G_2$
(Definition~\ref{def:ribbon-map}).
\begin{figure}
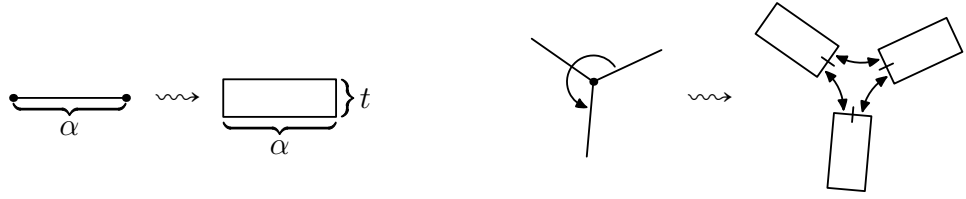

  \[
  \mfigb{surface-10}\;\rightsquigarrow\; \mfigb{surface-11}
  \qquad\qquad\qquad
  \mfigb{surface-0}\;\rightsquigarrow\; \mfigb{surface-2}
  \]
  \caption{Geometrically thickening an elastic ribbon graph. Left: An
    edge of elastic
    length~$\alpha$ is thickened to an $\alpha \times \eps$
    rectangle. Right: At a vertex, glue each half of the end of each
    rectangle to one of the neighbors according to the ribbon structure.}
  \label{fig:geom-thicken}
\end{figure}

\begin{theorem}\label{thm:embedding-thickening}
  Let $G$ and~$H$ be two elastic ribbon graphs with
  only trivalent vertices, and let $\phi \co G \to H$ be a
  ribbon map between them. Let $m$ be the minimal value
  of $\alpha(e)$ for $e$ an edge in~$G$ or~$H$. Then, for $\eps < m/2$,
  \[
  \SF[\phi]/(1+8\eps/m) \le
    \SF[N_\eps\phi] \le
    \SF[\phi]\cdot (1 + 8\eps/m).
  \]
\end{theorem}

Theorem~\ref{thm:embedding-thickening} is proved in
Section~\ref{sec:thicken}. We can get some intuition for
Theorem~\ref{thm:embedding-thickening} from a corollary,
which motivates the term ``embedding energy'' but is not
otherwise used.

\begin{corollary}\label{cor:asymp-embed}
  Let $G_1$ and~$G_2$ be two elastic ribbon graphs with
  trivalent vertices, and let $\phi \co G_1 \to G_2$ be a ribbon map
  between them.
  Then if $\Emb[\phi] < 1$, for all sufficiently small~$\eps$ there is
  a conformal embedding in $[N_\eps\phi]$. On the other hand, if for
  some sufficiently
  small~$\eps$ there is a conformal embedding in
  $[N_\eps\phi]$, then $\Emb[\phi] \le 1$.  
\end{corollary}

\begin{proof}
  Immediate from Theorem~\ref{thm:embedding-thickening},
  Theorem~\ref{thm:emb-sf}, Proposition~\ref{prop:sf-gen-simp}, and
  Theorem~\ref{thm:emb-surf}.
\end{proof}

\dinkus

The fifth and final ingredient is a study of the behavior of the
stretch factors and embedding energy under iteration.  Let
$f \co (S^2,P) \righttoleftarrow$ be a branched self-cover. Then we
have a virtual endomorphism $\pi,\phi \co X_1 \rightrightarrows X_0$,
where the $X_i$ are either Riemann surfaces or elastic
graphs. By iterating, we get a sequence of virtual
endomorphisms $\pi^n, \phi^n \co X_n \rightrightarrows X_0$, each with
its own stretch factor $\SF[\phi^n]$ (Definition~\ref{def:energy-iterate}). From general
principles (Proposition~\ref{prop:asympt-e}), it is not hard to prove
that the stretch factor grows or shrinks exponentially.
That is,
\begin{equation}
\ASF[\pi,\phi] \coloneqq \lim_{n \to \infty} \sqrt[n]{\SF[\phi^n]}
\end{equation}
exists, whether we are
working with elastic graphs or with Riemann surfaces. We call this
limit the \emph{asymptotic stretch factor}. General
principles also show that $\ASF[\pi,\phi]$ doesn't depend on the
particular conformal or elastic structure we start with
(Proposition~\ref{prop:asympt-homotopy}), and furthermore
Theorem~\ref{thm:embedding-thickening}
implies that $\ASF[\pi,\phi]$ is the same in these two
cases (Proposition~\ref{prop:asf-graph-surface}).
Theorem~\ref{thm:rational-surfaces-embed} and Theorem~\ref{thm:sf-cover}
(a strengthening of Theorem~\ref{thm:emb-surf}) then show that
$\ASF[\pi,\phi] < 1$ iff $f$ is
equivalent to a rational map. This is then translated to a proof of
Theorem~\ref{thm:detect-rational}, as explained in
Section~\ref{sec:iter}.

\dinkus

The last two sections have material complementary to the main
proof. In Section~\ref{sec:other-energies}, we explain how a virtual
endomorphism $(\pi,\phi)$ gives an asymptotic energy
$\oE^p_p[\pi,\phi]$ for each~$p$, with $\oE^2_2$ essentially agreeing with
$\ASF[\pi,\phi]$. We give bounds on how $\oE^p_p$ varies as a function
of~$p$ and a few comments on what this might mean.

In Section~\ref{sec:obstructions} we compare
Theorem~\ref{thm:detect-rational} with the original characterization,
Theorem~\ref{thm:thurston-obstruction}.

\subsection{Prior work}
\label{sec:prior}

Kahn's work on degenerating surfaces \cite[Section~3]{Kahn06:BoundsI}
has close ties to this work. In particular, his work is quite similar
in spirit to
Corollary~\ref{cor:asymp-embed}. His notion of \emph{domination} of
weighted arc diagrams is equivalent to embedding
energy being less than one, in the following dualizing sense.
Given a elastic ribbon graph $G = (\Gamma, \alpha)$, each edge~$e$
of~$\Gamma$ has a dual arc~$A_e$ on $N\Gamma$, the arc between
boundary components that meets~$e$ in one point. We can then consider
the dual weighted arc diagram
\[
W_G \coloneqq \sum_{e \in \Edges(\Gamma)} \frac{A_e}{\alpha(e)}.
\]
\begin{proposition}\label{prop:emb-dual}
  If $\phi \co G_1 \to G_2$ is a ribbon map of elastic ribbon graphs, then
  $\Emb[\phi] \le 1$ iff $W_{G_1}$ dominates $W_{G_2}$.
\end{proposition}
The proof follows from tracing through the definitions of both notions.

The notion of weighted arc diagrams is a little more general than
elastic ribbon graphs, as only weighted arc diagrams for which the set of
arcs is filling can be written as~$W_G$
for some~$G$.  See also
\cite[Section~\ref*{summary:sec:dynam-teich}]{Thurston16:RubberBands}.

\medskip

The graphical description of branched self-covers in
Section~\ref{sec:spines} is closely related to the description in
terms of
automata and bisets by Nekrashevych. More
specifically, these graphical descriptions are examples of his
combinatorial models for expanding dynamical
systems~\cite{Nekrashevych14:CombModel}. (Nekrashevych also allows
models with
higher-dimensional cells. See also Section~\ref{sec:future}.)

\medskip

Theorem~\ref{thm:rational-surfaces-embed} has been circulating in the
community. The non-trivial direction is written as \cite[Theorem
5.2]{CPT16:Renorm} or \cite[Theorem 7.1]{Wang14:DecompositionHerman}.

\medskip

The overall plan of this argument was summarized earlier
\cite{Thurston16:RubberBands}, and some of the arguments were sketched
there. For completeness, the logically necessary arguments are
reproduced and expanded here.

\subsection{Future directions}
\label{sec:future}

Theorem~\ref{thm:detect-rational} raises several questions and future
directions.

First, there is the question of the iteration ($n$ in the theorem
statement), and whether it is strictly necessary to iterate to give a
positive certificate for any given rational map. It is first of all
clear that the necessary value of~$n$ depends on the elastic graph
spine~$G$ (including a dependence on the elastic constants). See
\cite[Figure~2]{Thurston16:RubberBands} for one concrete example. So
then the question becomes whether we can always find an elastic
spine~$G$ so that
$n=1$ suffices. There are some
rational where $n=1$ does \emph{not} suffice for any~$G$ by a case
analysis. These
include the barycentric subdivision rational map \cite[Figure
25]{CFKP03:Subdivision},
\[
  f(z) = \frac{4}{27}\frac{(z^2-z+1)^3}{(z(z-1))^2},
\]
and rational map \#3.2 in the census of quadratic rational maps with
at most $4$ post-critical points \cite{BBLPP00:CensusRational},
\[
  f(z) \approx \frac{13.531903z-13.531903}{(z+2.3829758)^2}.
\]
Analysis of the first case is made easier using the symmetries of the
map. Analysis of the second case is more involved, and we do not give
the details. In both these cases, the Julia set is a Sierpiński
carpet. (The second example is the unique quadratic map in that census
with a
Sierpiński carpet Julia set.)

For many other examples, it seems that $n=1$ does suffice. In
particular, a crucial role appears to be played by the \emph{crochet
  maps}, rational maps where two Fatou components can be
connected by a path that intersects the Julia set in only countably
many points. (The terminology was introduced by Dudko, Hlushchanka,
and Schleicher, in ongoing work.) They are also conjecturally the
rational maps for which the Ahlfors regular conformal dimension is
equal to~$1$.
\begin{conjecture}
  For any crochet post-critically finite rational map, $n=1$ suffices
  in Theorem~\ref{thm:detect-rational}: There is an elastic
  (orbi)graph spine~$G$ for which $\Emb_G[\phi] < 1$ without the need
  to iterate.
\end{conjecture}
\begin{question}
  For crochet rational maps, is there a preferred ``best'' spine for
  $\CCa \setminus P$? For polynomials, there is a slight modification
  of the Hubbard tree that serves this purpose \cite[Definition
  8.8]{Thurston16:RubberBands}, and in many cases there appear to be
  good candidates.
\end{question}

\medskip

Beyond this, one might hope for several extensions of
Theorem~\ref{thm:detect-rational}, in a variety of different
directions. We briefly survey some of them.

First, it is possible to relax the restriction to maps of hyperbolic
type in Theorem~\ref{thm:detect-rational}, to allow maps of
non-compact type.
As stated, the theorem is false in this generality. Consider, for
example, the dendritic polynomial $z^2+i$, with
graph virtual endomorphism
$\pi,\phi \colon \Gamma_1 \rightrightarrows \Gamma_0$
  \[
    \mathcenter{\begin{tikzpicture}
      \node (G1) at (0cm, 0cm) {$\mfigb{graphs-51}$};
      \node (G0) at (3cm, 0cm) {$\mfigb{graphs-50}$};
      \draw[bend left=15, ->] (G1.5) to node[above,cdlabel]{\pi} (G0.172);
      \draw[bend right=15, ->] (G1.-5) to node[below,cdlabel]{\phi} (G0.188);
    \end{tikzpicture}}.
  \]
The black loops around the marked points are mapped to themselves by
degree~$1$. This persists under covers, so $\ASF[\pi,\phi] = 1$ and we
do not have the strict inequality we need for
Theorem~\ref{thm:detect-rational}. Essentially, with our definition
of spines, the loops around these marked points count as ``Levy
cycles'', obstructing our criterion.

A proper treatment of this family of maps uses \emph{orbigraphs},
spaces locally modeled on a graph modulo a finite group, so that the
graph virtual endomorphism in this example becomes
  \[
    \mathcenter{\begin{tikzpicture}
      \node (G1) at (0cm, 0cm) {$\mfigb{graphs-53}$};
      \node (G0) at (3cm, 0cm) {$\mfigb{graphs-52}$};
      \draw[bend left=15, ->] (G1.5) to node[above,cdlabel]{\pi} (G0.172);
      \draw[bend right=15, ->] (G1.-5) to node[below,cdlabel]{\phi} (G0.188);
    \end{tikzpicture}},
  \]
where the ``$2$'' mark an orbifold point of order 2, the quotient of
an edge by an involution. Details will appear in a future paper; see
also \cite[Problem 8.22]{Thurston16:RubberBands}.

On the other hand, any graph-based criterion is unlikely to work when
the branched
self-cover $f$ has \emph{no} cycles with branch points. For rational
maps, these are the cases when the Julia set is the whole Riemann sphere~$\CCa$.
The issue is that for a graph virtual endomorphism that is \emph{contracting}
(i.e., in the language of Section~\ref{sec:other-energies},
$\oE^\infty_\infty[\pi,\phi] < 1$), by a result of
Nekrashevych~\cite{Nekrashevych14:CombModel} the Julia set is
homeomorphic to the
inverse limit of graphs $\Gamma_n$ with respect to the
maps~$\phi^n_{n-1}$. (See Definition~\ref{def:correspondence} for
terminology.) But any inverse limit of graphs has topological
dimension~$1$, and so cannot be homeomorphic to~$\CCa$.

\medskip

One might also ask for a generalization of
Theorem~\ref{thm:detect-rational} to allow the
post-critical set~$P$ to be infinite, probably with some other
restrictions. For W. Thurston's original theory, this a fruitful area
of research, with a series of papers
\cite{CJS03:GeomFiniteI,CJS03:GeomFiniteII,JZ09:SubHyp,CT11:CharHyperbolic,CT18:HypParDef}
leading ultimately to an obstruction theorem for
maps where the accumulation set of~$P$ is finite (the \emph{geometrically
finite maps}).

In another direction, exponential maps and other transcendental maps
$\CC \to \CC$ have been a fruitful area of research, with substantial
more complexity than the rational map case. In particular, the natural
analog of Question~\ref{quest:pcf} is almost entirely open. This deals
with topological maps of transcendental type that are post-singularly
finite (to include the forward orbits of asymptotic values). For the
obstruction criterion, there is a result for the exponential family
\cite{HSS09:Exponential}, but this is a special case and it is not
clear what to expect in general. It is natural to ask whether
there is any analogue of the positive criterion presented in this
paper.

In both these cases (geometrically finite or transcendental maps), any
graphical model needs to allow for infinite graphs: either because the
post-critical set is infinite (and we want some sort of spine for its
complement), or because the degree of the cover is infinite. Some of
the ingredients in this paper carry over without issue. For instance,
Propositions~\ref{prop:el-graph-surface}
and~\ref{prop:el-surface-graph}, relating extremal length on graphs
and their thickening to a surface, hold for infinite graphs. (Indeed,
Theorem~\ref{thm:embedding-thickening} can be generalized
considerably, to allow grafting along arbitrary embedded arcs and/or
circles, with some weakening of the conclusion.)
Likewise the quasi-conformal surgery techniques
recalled in Section~\ref{sec:quasi-conf-surg} have been well-studied
in these more general contexts.

Other ingredients appear harder to generalize to the setting of
infinite graphs (or infinitely-generated~$\pi_1$). For instance, in
earlier work, we gave several equivalent conditions for conformal
embeddings of Riemann surfaces of finite type with some ``space''
around them \cite{KPT15:EmbeddingEL}. These conditions are unlikely to
be equivalent in the setting of arbitrary Riemann surfaces.

Another possible generalization is to higher-dimensional maps, for
instance studying maps of higher-dimensional manifolds that are
post-critically finite in a suitable sense. It is not entirely clear
what the right questions or conjectures should be, but one can often
find suitable combinatorial models for such maps as virtual
endomorphisms of CW complexes. Nekrashevych has both a general theory
\cite{Nekrashevych14:CombModel} and concrete examples
\cite{Nekrashevych16:PaperFolding}. The major obstacle to developing
a theory similar to the one in this paper in the higher-dimensional
setting is finding the right analogue of the energies~$E^p_p$ recalled
in Section~\ref{sec:other-energies}. The definitions rely heavily on
the underlying objects being 1-dimensional. (See \cite[Definition
\ref{Elast:def:Epq-3}]{Thurston19:Elastic} for a possible direction
towards a generalization.) Furthermore, the proof of
existence of minimizers from earlier work is combinatorial and
produces piecewise-linear minimizing maps; this approach will not work
in higher dimensions. This remains work in progress.

\medskip

Another direction is investigating the meaning of the additional
energies $\oE^p_p$ in Section~\ref{sec:other-energies}. In work in
progress with Kevin Pilgrim, we relate these energies to the Ahlfors
regular conformal dimension~\cite{PT:ConfDim}.

\subsection*{Acknowledgements}

This project grew out of joint work with Kevin Pilgrim, and owes a great
deal to conversations with him and with Jeremy Kahn.
In addition, there were many helpful conversations and comments from
Maxime Fortier-Bourque, 
Frederick Gardiner, 
Mikhail Hlushchanka,
Volodymyr Nekrashevych,
Tan Lei, and the anonymous referee.

This material is based upon work supported by the National Science
Foundation under Grant Number DMS-1507244.


\section{Spines for branched self-covers}
\label{sec:spines}

\begin{definition}
  A \emph{virtual endomorphism} of a group $G$ is a finite-index
  subgroup $H \subset G$ and a homomorphism $\phi \co H \to G$.

  A \emph{virtual endomorphism} of a topological space~$X$ consists of
  a space~$Y$ and a pair of maps
  \[
  \pi, \phi \co Y \rightrightarrows X
  \]
  where $\pi$ is a covering map of constant, finite degree and $\phi$
  is considered
  up to homotopy.
\end{definition}

A virtual endomorphism of spaces
gives a virtual
endomorphism of groups, as follows. Suppose $X$ and~$Y$ are connected
and locally
connected and $x_0 \in X$
is a basepoint. If we pick $y_0 \in \pi^{-1}(x_0)$, then
$\pi_1(Y,y_0)$ is naturally a
subgroup of~$\pi_1(X,x_0)$. If we
homotop~$\phi$ so that
$\phi(y_0) = x_0$, then $\phi_*$ gives a group homomorphism from $\pi_1(Y,y_0)$ to
$\pi_1(X,x_0)$, i.e., a virtual endomorphism of $\pi_1(X,x_0)$.

Virtual endomorphisms of topological (orbi)spaces are also called
\emph{topological automata} by Nekrashevych
\cite{Nekrashevych14:CombModel}. If you drop the condition that $\pi$
be a covering map, the same structures were called \emph{topological
  graphs} or
\emph{topological correspondences} by Katsura
\cite{Katsura04:ClassCstarI} and \emph{multi-valued dynamical systems}
by Ishii and Smillie \cite{IS10:HomotopyShadowing}.

\begin{definition}\label{def:virt-homotopy-equiv}
  A \emph{homotopy morphism} between two virtual endomorphisms, from
  $\pi_X, \phi_X \co X_1 \rightrightarrows X_0$ to
  $\pi_Y, \phi_Y \co Y_1 \rightrightarrows Y_0$, is a pair of maps
  $f_0 \co X_0 \to Y_0$ and $f_1 \co X_1 \to Y_1$ so that
  \begin{align*}
    f_0 \circ \pi_X &= \pi_Y \circ f_1\\
    f_0 \circ \phi_X &\sim \phi_Y \circ f_1
  \end{align*}
  where $\sim$ means homotopy of maps.

  A \emph{homotopy equivalence} between $(\pi_X,\phi_X)$ and
  $(\pi_Y,\phi_Y)$ is a pair of homotopy morphisms $(f_0,f_1)$ from
  $X$ to~$Y$ and $(g_0,g_1)$ from $Y$ to~$X$, so that
  $f_0 \circ g_0 \sim \id_{Y_0}$ and $g_0 \circ f_0 \sim \id_{X_0}$.
  This implies that $f_1 \circ g_1 \sim \id_{Y_1}$ and
  $g_1 \circ f_1 \sim \id_{X_1}$, as shown below.
  \[
  \begin{tikzpicture}[x=2cm,y=2cm]
    \node (X1) at (0,0) {$X_1$};
    \node (X0) at (1,0) {$X_0$};
    \node (X1') at (0,-1) {$Y_1$};
    \node (X0') at (1,-1) {$Y_0$};
    \draw[->,bend left=15] (X1) to node[above,cdlabel]{\pi_X} (X0);
    \draw[->,bend right=15] (X1) to node[below,cdlabel]{\phi_X} (X0);
    \draw[->,bend left=15] (X1') to node[above,cdlabel]{\pi_Y} (X0');
    \draw[->,bend right=15] (X1') to node[below,cdlabel]{\phi_Y} (X0');
    \draw[->,bend left=20] (X0) to node[right,cdlabel]{f_0} (X0');
    \draw[->,bend left=20] (X0') to node[left,cdlabel]{g_0} (X0);
    \draw[->,bend left=20] (X1) to node[right,cdlabel]{f_1} (X1');
    \draw[->,bend left=20] (X1') to node[left,cdlabel]{g_1} (X1);
    \node at (0,-0.5) {$\sim$};
    \node at (1,-0.5) {$\sim$};
  \end{tikzpicture}
  \]
\end{definition}

If $f \co (\Sigma, P) \righttoleftarrow$ is a branched self-cover of a
surface, let $\Sigma_0 = \Sigma \setminus P$ and
$\Sigma_1 = \Sigma \setminus f^{-1}(P)$. The restriction of~$f$ gives
a covering map $\pi_\Sigma \co \Sigma_1 \to \Sigma_0$, and the natural
inclusion of surfaces gives a map
$\phi_\Sigma \co \Sigma_1 \to \Sigma_0$, together forming a surface virtual
endomorphism
\begin{equation}\label{eq:surf-end}
\pi_\Sigma, \phi_\Sigma \co \Sigma_1 \rightrightarrows \Sigma_0.
\end{equation}

A \emph{spine} of~$\Sigma_0$ is a graph $\Gamma_0 \subset \Sigma_0$ that is
a deformation retract of~$\Sigma_0$. If we replace $\Sigma_0$ in
\eqref{eq:surf-end} by a
spine~$\Gamma_0$, we get spaces and maps
\begin{itemize}
\item $\Gamma_1 = f^{-1}(\Gamma_0)  \subset \Sigma_1$;
\item deformation retractions $\kappa_i \co \Sigma_i \to \Gamma_i$;
\item the restriction of $f$ to a covering of graphs $\pi_\Gamma \co \Gamma_1 \to \Gamma_0$; and
\item $\phi_\Gamma = \kappa_0 \circ \phi_\Sigma \co \Gamma_1 \to \Gamma_0$.
\end{itemize}
These form a graph virtual endomorphism
\begin{equation}\label{eq:graph-end}
  \pi_\Gamma, \phi_\Gamma \co \Gamma_1 \rightrightarrows \Gamma_0.
\end{equation}
Since the $\kappa_i$ are homotopy equivalences, $[\phi_\Gamma]$ is
determined by $[\phi_\Sigma]$.
While $\phi_\Sigma$ is a topological inclusion,
$\phi_\Gamma$ is just a continuous map of graphs. We say $(\pi_\Gamma, \phi_\Gamma)$ is \emph{compatible}
with the branched self-cover~$f$.
Since any two spines for~$\Sigma_0$ are homotopy equivalent, the
homotopy equivalence
class $[\pi_\Gamma, \phi_\Gamma]$ is determined by~$f$.

To go the other direction and recover the branched self-cover from the graph
virtual endomorphism
$\pi_\Gamma, \phi_\Gamma \co \Gamma_1 \rightrightarrows \Gamma_0$, we
need some more
data.

\begin{definition}\label{def:ribbon-graph}
  A \emph{ribbon structure} on a graph~$\Gamma$ is, for each
  vertex~$v$ of~$\Gamma$, a cyclic ordering on the ends of edges
  incident to~$v$, thought of as the counterclockwise ordering. A
  ribbon structure gives a canonical
  thickening of~$\Gamma$ into an oriented surface with
  boundary~$N\Gamma$, the underlying topological surface of the
  Riemann surface $N_\eps\Gamma$ from Figure~\ref{fig:geom-thicken}.
  There is a natural inclusion
  $i_{N\Gamma} \co \Gamma \hookrightarrow N\Gamma$ and projection
  $\pi_{N\Gamma} \co N\Gamma \rightarrow \Gamma$.
\end{definition}

We will prove that a virtual endomorphism of a ribbon graph is
compatible with
at most one branched self-cover.

For an example of what this data looks like, consider the rational
map
\[f(z) = (1+z^2)/(1-z^2),\]
with critical portrait
\begin{equation}\label{eq:crit-portrait}
\raisebox{7pt}{$\mathcenter{\begin{tikzpicture}[node distance=1.75cm]
  \node (0) at (0,0) {$0$};
  \node (1)[right of=0] {$1$};
  \node (inf)[right of=1] {$\infty$};
  \node (n1)[right of=inf] {$-1.$};
  \draw[->] (0) to node[above,cdlabel]{(2)} (1);
  \draw[->] (1) to (inf);
  \draw[->,bend left] (inf) to node[above,cdlabel]{(2)} (n1);
  \draw[->,bend left] (n1) to (inf);
\end{tikzpicture}}$}
\end{equation}
We take $P$ to be the post-critical set $\{-1,1,\infty\}$. We can
take $\Gamma_0$ to be a $\Theta$-graph embedded in
$\Sigma_0 = S^2\setminus P$ and take $\Gamma_1$ to be $f^{-1}(\Gamma_0)$, as
indicated in Figure~\ref{fig:repr-spine}. The map~$\pi$ is the
covering map that preserves labels and orientations on the edges. The
map $\phi$
might, for instance, be chosen so that
\begin{itemize}
\item the two $a$ edges of $\Gamma_1$ map to the $a$ and $b$ edges
  of~$\Gamma_0$;
\item the two $b$ edges of $\Gamma_1$ map to the $c$ edge
  of~$\Gamma_0$; and
\item the two $c$ edges of $\Gamma_1$ map with a constant map to the two vertices
  of~$\Gamma_0$.
\end{itemize}

\begin{figure}
  \centering
  \subcaptionbox{Virtual endomorphism of a
    spine for $S^2 \setminus
    \{-1,1,\infty\}$\label{fig:repr-spine}. The marked point $\infty$
    is at infinity.}
    [.4\linewidth]{
    \begin{tikzpicture}[baseline]
      \node (G1) at (0cm,1.4cm) {$\mfigb{repr-2}$};
      \node (G0) at (0cm,-1.4cm) {$\mfigb{repr-1}$};
      \draw[bend right,->] (G1.-95) to node[left=-2pt,cdlabel]{\phi} (G0.95);
      \draw[bend left,->] (G1.-85) to node[right=-2pt,cdlabel]{\pi} (G0.85);
    \end{tikzpicture}}\qquad
  \subcaptionbox{Virtual endomorphism of a rose
    graph.\label{fig:repr-rose}}
    [.4\linewidth]{
    \begin{tikzpicture}[baseline]
      \node (G1) at (0cm,1.4cm) {$\mfigb{repr-4}$};
      \node (G0) at (0cm,-1.4cm) {$\mfigb{repr-3}$};
      \draw[bend right,->] (G1.-95) to node[left=-2pt,cdlabel]{\phi} (G0.95);
      \draw[bend left,->] (G1.-85) to node[right=-2pt,cdlabel]{\pi} (G0.85);
    \end{tikzpicture}}\\[12pt]
  \subcaptionbox{Dual Moore diagram\label{fig:repr-dual-moore}. $1$ is
    the identity element in the group $F_2 = \langle a,b\rangle$, and
    capital letters denote inverses.}
    [.4\linewidth]{$\mfigb{repr-5}$}\qquad
  \subcaptionbox{Textual description of automaton.\label{fig:repr-text}}
    [.4\linewidth]{
      $\begin{aligned}
        a(0\cdot w) &= 1\cdot b(w)\\
        a(1\cdot w) &= 0\cdot A(w)\\
        b(0\cdot w) &= 1\cdot w\\
        b(1\cdot w) &= 0\cdot w
      \end{aligned}$}
  \caption{Representations of the rational map
    $z \mapsto (1+z^2)/(1-z^2)$.}
  \label{fig:representations}
\end{figure}
To read off the critical portrait, first recall that from a connected
ribbon graph
embedded in the plane, the complementary regions
are intrinsically determined by
following the boundary of the ribbon surface. Thus we can talk about
the regions of $\Gamma_0$ and~$\Gamma_1$.
Then, for instance, the point ``$1$'' is in the region of~$\Gamma_1$
surrounded by an $a$ edge and a $b$ edge, so must map by~$f$ to
the point ``$\infty$'', which is in the exterior region of~$\Gamma_0$,
also surrounded by an $a$ edge and a $b$ edge. On the other hand,
``$\infty$'' in the exterior
region of
$\Gamma_1$ is surrounded by an $a$, $c$, $a$, and~$c$
edge, and so maps with double branching to ``$-1$'', in the region of~$\Gamma_0$
surrounded by $a$ and~$c$. Proceeding in
this way, we recover the critical portrait~\eqref{eq:crit-portrait}.

This data is
essentially equivalent to an automaton in the style of
Nekrashevych \cite{Nekrashevych05:SelfSimilar}. To construct the
automaton, first choose a
spanning tree~$T_0$ inside~$\Gamma_0$ and collapse it to get a
rose graph spine~$R_0$ for~$\Sigma_0$. (A \emph{rose graph} is a graph
with one vertex.) If we collapse
$\pi^{-1}(T_0)$ inside $\Gamma_1$, we get $R_1$, which is likewise a
spine for~$\Sigma_1$. ($R_1$ is not itself a rose graph.)
Since $R_0$ is also a spine for $S^2 \setminus P$, there is a virtual
endomorphism $\pi_R, \phi_R \co R_1 \rightrightarrows R_0$.
In the running example, if we take the spanning tree of~$\Gamma_0$ to
be edge~$c$ in
Figure~\ref{fig:repr-spine}, we get the graphs $R_0$ and~$R_1$ in
Figure~\ref{fig:repr-rose}.

The graph $R_1$ constructed above is quite close to the dual Moore
diagram for the corresponding automaton, shown in
Figure~\ref{fig:repr-dual-moore}
for the running example.
To get from Figure~\ref{fig:repr-rose} to
Figure~\ref{fig:repr-dual-moore}, perform the following steps.
\begin{enumerate}
\item Homotop the graph map $\phi \co R_1 \to R_0$ so that it
  sends vertices to the vertex of~$R_0$. In the example, the two $b$
  edges of~$R_1$ get mapped to points.
\item As a graph, the dual Moore diagram~$D$ is~$R_1$, with vertices
  numbered arbitrarily.
\item Label each edge~$e$ of $D$ by, first, the label of~$e$ in $R_1$ and,
  second, the element of $\pi_1(R_0)$ represented by $\phi(e)$.
\end{enumerate}
The dual Moore diagram encodes an
automaton, which in the example is given textually in
Figure~\ref{fig:repr-text}.

\medskip

Returning to the general theory, not all combinations of a graph virtual
endomorphism
$\pi,\phi \co \Gamma_1 \rightrightarrows \Gamma_0$ and a ribbon graph
structure on~$\Gamma_0$ are compatible with a branched
self-cover.

\begin{definition}\label{def:ribbon-map}
  If $\Gamma$ and $\Gamma'$ are ribbon graphs, a \emph{ribbon map}
  $\phi \co \Gamma \to \Gamma'$ is a map that lifts to
  an orientation-preserving topological embedding
  $N\phi \co N\Gamma \hookrightarrow N\Gamma'$, in the sense that
  $\phi = \pi_{N\Gamma'} \circ N\phi \circ i_{N\Gamma}$.
\end{definition}

\begin{lemma}\label{lem:ribbon-lift}
  If $\phi \co \Gamma \to \Gamma'$ is a map between ribbon graphs and $\Gamma$
  and~$\Gamma'$ are connected, then up to isotopy there is at most one
  orientation-preserving lift $N\phi \co N\Gamma \to N\Gamma'$.
\end{lemma}

\begin{proof}
  This follows from the fact that any two
  orientation-preserving homotopic
  embeddings from one connected surface to another are isotopic, which
  in turn follows from work of Epstein
  \cite{Epstein66:Curves2Manifolds} by looking at the
  boundary curves \cite{Putman16:IsotopSurface}. It is also proved
  as a side effect of work of Fortier Bourque on conformal
  embeddings~\cite{FB18:Couch}.
\end{proof}

\begin{definition}
  Suppose that $\pi, \phi \co \Gamma_1 \rightrightarrows \Gamma_0$ is
  a graph virtual endomorphism where $\Gamma_0$ has a ribbon graph
  structure. We can use the covering map~$\pi$ to pull back the ribbon
  structure on~$\Gamma_0$ to a ribbon structure on~$\Gamma_1$. Then we
  say the data form a \emph{ribbon virtual endomorphism} if
  $\phi$ is a ribbon map. A
  \emph{ribbon homotopy morphism} between two
  ribbon virtual
  endomorphisms is a homotopy morphism as in
  Definition~\ref{def:virt-homotopy-equiv} so that~$f_0$ and $f_1$ are
  ribbon maps.
\end{definition}

\begin{remark}
  It is not immediately clear how to give an efficient algorithm to
  check whether a topological map $\phi \co \Gamma \to \Gamma'$
  between ribbon graphs
  is a ribbon map, but we can give an inefficient algorithm.
  If we specify, for each regular point
  $y \in \Gamma'$, the order in which the points in $\phi^{-1}(y)$
  appear on the corresponding cross-section of $N\Gamma'$, it is easy
  to check locally whether there is an embedded lift. Since
  there are only finitely many choices of orders, this can be
  checked algorithmically.
\end{remark}

\begin{definition}
  A map $\phi \co X \to Y$ between locally path-connected topological
  spaces is \emph{$\pi_0$-bijective} if it gives a bijection from the
  connected components of~$X$ to the connected components
  of~$Y$. (Recall that branched self-covers are assumed to be
  $\pi_0$-bijective.) The map~$\phi$
  is \emph{$\pi_1$-surjective} if, for each $x \in X$, the induced map
  $\phi_* \co \pi_1(X, x) \to \pi_1(Y,\phi(x))$ is surjective.
\end{definition}

\begin{theorem}\label{thm:spine-surf}
  Branched self-covers of surfaces
  $f \co (\Sigma, P) \righttoleftarrow$, up to equivalence, are in
  bijection with ribbon virtual endomorphisms
  $\pi,\phi \co \Gamma_1 \rightrightarrows \Gamma_0$ so that $\phi$ is $\pi_0$-bijective and
  $\pi_1$-surjective, up to ribbon homotopy equivalence.
\end{theorem}

\begin{remark}
  Theorem~\ref{thm:spine-surf} does not assume \emph{a priori} that
  the surface $\Sigma$ is a (union of) spheres or
  that the $\Gamma_i$ are planar. Once the theorem is proved
  then the usual Euler characteristic arguments imply that each
  component of~$\Sigma$ is a sphere or a torus, with sphere being by
  far the more interesting case.
\end{remark}

\begin{proof}
  If we are given a branched self-cover $f \co (\Sigma, P)
  \righttoleftarrow$, we have already seen how to pick a compatible spine
  $\Gamma_0 \subset \Sigma \setminus P$ and construct a ribbon virtual
  endomorphism $\pi, \phi \co \Gamma_1 \rightrightarrows \Gamma_0$,
  unique up to ribbon homotopy equivalence. It is immediate that
  $\phi$ is $\pi_0$-bijective and $\pi_1$-surjective.

  It remains to check the other direction. Suppose we have a ribbon
  virtual endomorphism as in the statement.
  Let $\Sigma_0 = N\Gamma_0$ and $\Sigma_1 = N\Gamma_1$. Since $\phi$
  is $\pi_0$-bijective,
  Lemma~\ref{lem:ribbon-lift} tells us the lift $N\phi \co \Sigma_1
  \hookrightarrow \Sigma_0$ is unique.
  Let $\wh{\Sigma}_0$ be the marked
  surface obtained
  by attaching a disk with a marked
  point in the center to each boundary component of~$\Sigma_0$. Let
  $P_0\subset \wh\Sigma_0$ be the set of marked points.

  Recall that a simple closed curve~$C$ on a closed surface is
  separating iff it is homologically trivial, and that it is
  non-separating iff there is a ``dual'' simple closed curve~$C'$
  that intersects~$C$ transversally in one point.

  Let $C_1$ be a boundary component of~$\Sigma_1$ and consider the
  simple curve $C_0 = N\phi(C_1) \subset \Sigma_0 \subset \wh\Sigma_0$. If
  $C_0$ is non-separating, a dual curve
  cannot be homotoped to lie in the
  image of~$N\phi$, contradicting $\pi_1$-surjectivity. So $C_0$ is
  separating and divides $\wh\Sigma_0$ into two components,
  with one component containing the image of~$N\phi$. If the other
  component is not a disk with $0$ or $1$ marked
  points, then again $\phi$ is not $\pi_1$-surjective. If $N\phi(C_1)$ bounds
  a disk with no marked points in~$\wh\Sigma_0$ (so bounds a disk
  in~$\Sigma_0$), say that $C_1$ is \emph{collapsed}.

  Now construct~$\wh\Sigma_1$ by attaching a disk~$D_i$ to
  each boundary component~$C_i$ of~$\Sigma_1$. Mark the center of $D_i$
  if $C_i$ is not
  collapsed. Let
  $P_1 \subset \wh\Sigma_1$ be the set of marked points. By the
  choices made in the construction, $N\phi$ extends to a
  homeomorphism $g \co \wh\Sigma_1 \to \wh\Sigma_0$ inducing a bijection
  from $P_1$ to~$P_0$.

  Since the ribbon structure on~$\Gamma_1$ is the pull-back of
  the ribbon structure on~$\Gamma_0$, the covering map~$\pi$
  extends to a covering of surfaces
  $N\pi \co \Sigma_1 \to \Sigma_0$. Since $N\pi$ restricts to a
  covering map from $\bdy N\Gamma_1$ to $\bdy N\Gamma_0$, we can
  extend $N\pi$ to a branched cover
  $h \co \wh\Sigma_1 \to \wh\Sigma_0$ with $h(P_1) \subset P_0$ and
  branch values contained in~$P_0$.

  The desired branched self-covering is then
  $f = h \circ g^{-1} \co (\wh\Sigma_0,P_0) \righttoleftarrow$. The original virtual endomorphism
  $\pi,\phi \co \Gamma_1 \rightrightarrows \Gamma_0$ is compatible with~$f$.
\end{proof}

\begin{example}\label{examp:rabbits}
  To see that the ribbon structure is necessary in
  Theorem~\ref{thm:spine-surf}, consider the $1/5$
  and $2/5$ rabbit (i.e., the centers of the $1/5$ and $2/5$ bulb of
  the Mandelbrot set), with compatible graph virtual endomorphisms
  shown in
  Figure~\ref{fig:rabbit-spines}. The two branched self-covers are
  different, but the graph virtual endomorphisms are the same except
  for the ribbon structure.
\end{example}

\begin{figure}
  \[
  \begin{tikzpicture}
    \node at (0cm,-0.6cm) {$\mfig{graphs-81}$};
    \node at (3.6cm,+0.2cm) {$\mfig{graphs-80}$};
    \draw[bend left,->] (1.5cm,0.1cm) to node[above,cdlabel]{\pi} (2.1cm,0.1cm);
    \draw[bend right,->] (1.5cm,-0.1cm) to node[below,cdlabel]{\phi} (2.1cm,-0.1cm);
  \end{tikzpicture}
  \qquad\qquad
  \begin{tikzpicture}
    \node at (0cm,-0.6cm) {$\mfig{graphs-83}$};
    \node at (3.6cm,+0.2cm) {$\mfig{graphs-82}$};
    \draw[bend left,->] (1.5cm,0.1cm) to node[above,cdlabel]{\pi} (2.1cm,0.1cm);
    \draw[bend right,->] (1.5cm,-0.1cm) to node[below,cdlabel]{\phi} (2.1cm,-0.1cm);
  \end{tikzpicture}
  \]
  \caption{Spines for the $1/5$ rabbit (left) and $2/5$ rabbit
    (right). There are extra
    marked points at infinity. The map~$\pi$ is the cover
    that preserves colors/labels and the map~$\phi$ is
    determined by the deformation retraction.}
  \label{fig:rabbit-spines}
\end{figure}
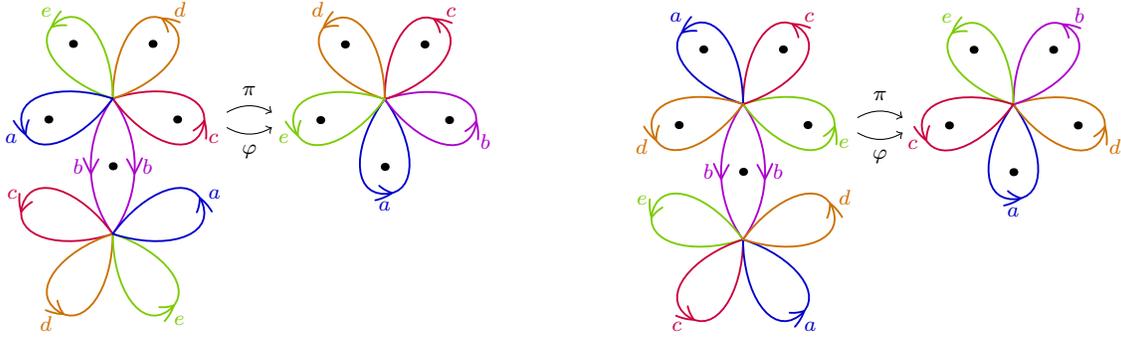

\begin{remark}
  The relationship with Nekrashevych's automata can be used to give
  another proof of the uniqueness part of
  Theorem~\ref{thm:spine-surf} \cite[Theorem
  6.5.2]{Nekrashevych05:SelfSimilar} .
\end{remark}


\section{Quasi-conformal surgery}
\label{sec:quasi-conf-surg}

We now turn to the (standard) characterization of rational maps in terms of
conformal embeddings of surfaces.
For this section, we generalize to maps of non-compact
type, since we can do it with little extra work.

Let $f \co (\Sigma,P) \righttoleftarrow$ be a branched self-cover of
non-compact type. Let $P_F$ (with $F$ for Fatou) be the set of points in~$P$
whose forward orbit under $f$ lands in a cycle with a branch point,
and let $P_J$
(with $J$ for Julia) be $P \setminus P_F$. Let $\Sigma_0$ be the complement of
a disk neighborhood of $P_F$ in~$\Sigma$, with marked subset
$P_{0J} \coloneqq P_J$. Parallel to \eqref{eq:surf-end}, there is a
\emph{branched} virtual endomorphism
\[\pi, \phi \co (\Sigma_1,P_{1J}) \rightrightarrows (\Sigma_0,P_{0J}),
\]
 where
$\Sigma_1 \coloneqq f^{-1}(\Sigma_0)$, $P_{1J}$ is $P_J$ as a subset
of~$\Sigma_1$, $\pi$ is a branched cover with branch
values contained in~$P_{0J}$, and
$\phi$ induces a bijection between $P_{1J}$ and~$P_{0J}$. We consider~$\phi$
up to homotopy relative to~$P_{1J}$.

Given a branched virtual endomorphism of surfaces
$\pi,\phi \co (\Sigma_1,P_{1J}) \rightrightarrows (\Sigma_0,P_{0J})$,
if there is a
conformal structure~$\omega_0$ on~$\Sigma_0$ there is a pull-back
conformal structure $\omega_1 \coloneqq \pi^* \omega_0$
on~$\Sigma_1$. Then $(\pi,\phi)$ is said to be \emph{conformal} with
respect to~$\omega_0$ if $\phi$ is homotopic (rel~$P_{1J}$) to
an annular conformal
embedding from $(\Sigma_1, \omega_1)$ to $(\Sigma_0,\omega_0)$.

\begin{definition}
  The \emph{Teichmüller space} of a branched virtual endomorphism of
  surfaces
  $\pi,\phi\co(\Sigma_1,P_{1J})\rightrightarrows(\Sigma_0,P_{0J})$
  is the space
  $\Teich(\pi,\phi)$ of isotopy classes of complex
  structures~$\omega_0$ on~$\Sigma_0$ so that $(\pi,\phi)$ is
  conformal with respect to~$\omega_0$.
\end{definition}

\begin{remark}
  In the literature, the condition that the embedding~$\phi$ be
  annular is often omitted.
\end{remark}

\begin{citethm}\label{thm:rational-surfaces-embed}
  Let $f \co (\Sigma, P) \righttoleftarrow$ be a branched self-cover
  of non-compact type, and let
  $\pi,\phi \co (\Sigma_1,P_{1J}) \rightrightarrows (\Sigma_0, P_{0J})$ be
  the associated branched virtual endomorphism. Then $f$ is equivalent
  to a rational map iff $\Teich(\pi,\phi)$ is non-empty.
\end{citethm}

\begin{proof}
  We start with the easy direction. If $f$ is equivalent to a rational
  map, replace it with its rational map
  $f \co (\CCa,P) \righttoleftarrow$. Let
  $J \subset \CCa$ be the Julia set of~$f$. Then $P \cap J =
  P_J$.
  Set $S_0$ to be a suitable open neighborhood of~$J$, chosen so that
  $\overline{f^{-1}(S_0)} \subset S_0$. Then we can take $S_1 =
  f^{-1}(S_0)$.

  To be concrete about the ``suitable open neighborhood'', let $F$ be
  the Fatou set of~$f$, and choose a Green's function on~$F$, by which
  we mean a harmonic
  function $G \co F \to (0,\infty]$ (taking the value $\infty$ only at
  discrete points) so that
  \[
  G(f(z)) = \lambda_z G(z)
  \]
  where $\lambda_z > 1$ is a locally constant function on~$F$ so that
  $\lambda_{f(z)} = \lambda_z$. (This implies $G$ goes to zero near
  $\partial F$.) Concretely, if~$z$ is attracted to a
  cycle in~$P_F$ of period~$p$ and total branching~$d$, then
  $\lambda_z = d^{1/p}$. On a basin~$B_i$ of~$F$ with Böttcher
  coordinate $\phi_i \co B_i \to \DD$, we can take
  $G(z) = -K_i \log\, \abs{\phi_i}$ on~$B_i$ for some constant~$K_i$.

  Extend $G$ to all of $\CCa$ by setting $G(z) = 0$ for $z \in J$.
  Pick $\eps > 0$, and take $S_0$ to be the union of
  $G^{-1}([0,\eps])$ and all basins in~$F$ that do not
  contain a point of~$P_F$. We then have a
  conformal branched virtual automorphism
  $\pi,\phi \co (f^{-1}(S_0),P_J) \rightrightarrows (S_0,P_J)$.

  The other direction is a special case of \cite[Theorem
  5.2]{CPT16:Renorm} or \cite[Theorem
  7.1]{Wang14:DecompositionHerman}. The technique is due to Douady
  and Hubbard
  \cite{DH85:DynamicsPolyLike}.
  We sketch the proof here.

  Suppose $\omega_0$ is a point in $\Teich(\pi,\phi)$, and consider
  the corresponding conformal maps
  $\pi,\phi \co (S_1,P_{1J}) \rightrightarrows (S_0,P_{0J})$. We extend $\pi$ and~$\phi$
  to maps
  $\wh\pi, \wh\phi \co (\wh S_1,P_1) \rightrightarrows (\wh S_0,P_0)$
  between closed surfaces as in
  the proof of
  Theorem~\ref{thm:spine-surf}, but with attention to keeping the maps
  conformal except in controlled ways.
  \begin{itemize}
  \item Let $\wh S_0$ be the compact Riemann surface
    without boundary
    obtained by attaching disks to the boundary components
    of~$S_0$. Let $D_0$ be the union of the new disks and let
    $V_0 \subset D_0$ be a union of concentrically
    contained smaller disks. Let $P_{0F}$ be the set of points in the
    center of each component of~$V_0$ (and~$D_0$). Let $P_0 = P_{0J} \sqcup P_{0F}$.
  \item Let $\wh S_1$ be the corresponding conformal branched cover of $\wh S_0$,
    branched at points in~$P_0$, with
    $\wh\pi \co \wh S_1 \to \wh S_0$ extending~$\pi$. Let $D_1$ be
    $\wh \pi^{-1}(D_0)$. Define $P_{1F} \subset \wh S_1$ by picking
    those points in $\wh\pi^{-1}(P_{0F})$ that are in the center of
    non-collapsed boundary components of~$S_1$, as in the proof of
    Theorem~\ref{thm:spine-surf}. Let $V_1 \subset D_1$ be the union of
    those components of $\wh\pi^{-1}(V_0)$ that contain a point
    in~$P_{1F}$ and let $P_1 = P_{1J} \sqcup P_{1F}$.
  \end{itemize}
  Next we will define a diffeomorphism
  $\wh\phi \co (\wh S_1,P_1) \to (\wh S_0,P_0)$ in stages.
  \begin{itemize}
  \item On $S_1$, the map $\wh\phi$ agrees with $\phi$.
  \item Each component of~$V_1$ contains a point $p_1 \in P_{1F}$, which
    corresponds to a point $p_0 \in P_{0F}$. On this
    component, $\wh\phi$ is a conformal identification of the disk of~$V_1$
    to the disk of~$D_0$ containing~$p_0$, mapping $p_1$ to~$p_0$.
  \item It remains to define $\wh\phi$ on $D_1 \setminus V_1$, which
    is a union of annuli and unmarked disks. Define $\wh\phi$ to be an
    arbitrary diffeomorphism extending the maps defined so far. This
    is possible since we haven't changed the isotopy class from the
    homeomorphism from Theorem~\ref{thm:spine-surf}.
  \end{itemize}
  Observe that $\wh\phi$ and~$\wh\phi^{-1}$ are $K$-quasi-conformal
  for some $K \ge 1$, since
  they are
  diffeomorphisms on compact manifolds.
  The map $\wh\phi$ takes the pieces of $\wh S_1$ to the pieces of $\wh S_0$:
  \[
  \begin{tikzpicture}[x=1.5cm,y=1.5cm]
    \node (S1) at (0,1) {$S_1$};
    \node (D1) at (1,1) {$D_1 \setminus V_1$};
    \node (V1) at (2,1) {$V_1$};
    \node (S0) at (0.5,0) {$S_0$};
    \node (D0) at (1.5,0) {$D_0$.};
    \node at ($ (S1.east)!0.5!(D1.west) $) {$\sqcup$};
    \node at ($ (D1.east)!0.5!(V1.west) $) {$\sqcup$};
    \node at ($ (S0.east)!0.5!(D0.west) $) {$\sqcup$};
    \draw[->] (S1) to node[below,sloped]{\small conf.} (S0);
    \draw[->] (D1) to node[below,sloped]{\small q.c.} (S0);
    \draw[->] (V1) to node[below,sloped]{\small conf.} (D0);
  \end{tikzpicture}
  \]

  Now let $f \co (\wh S_0, P_0) \righttoleftarrow$ be
  $\wh\pi \circ \wh\phi^{-1}$. Since $\wh\pi$ is conformal, $f$ is $K$-quasi-regular, $f$ is
  conformal on $\phi(S_1) \sqcup D_0$, and $f(D_0) = V_0 \subset D_0$.
  If $f$ is not conformal at $x \in \wh S_1$, then
  $f(x) \in D_0$, so $f$ is conformal at
  $f(x)$ and at all further forward iterates of $f(x)$. That is, in
  the forward orbit of any $x \in \wh S^1$, there is at most one~$n$
  so that $f$ is not conformal at $f^{\circ n}(x)$. Thus $f^{\circ n}$
  is also $K$-quasi-regular with the same value
  of~$K$. We can therefore apply Sullivan's Averaging Principle to
  find an invariant measurable complex structure on $\wh S_1$, which
  by the Measurable Riemann Mapping Theorem can be straightened to
  give an honest conformal structure and a post-critically finite
  map~\cite{Sullivan81:ErgInf}.
\end{proof}

\begin{remark}
  Theorem~\ref{thm:rational-surfaces-embed} is also true if we take
  $P_F$ to be an invariant subset of the points of~$P$ whose forward
  orbit lands in a branched cycle, as long as there is at least one
  point of $P_F$ in each component of $\Sigma$. For instance, for
  polynomials we may take $P_F = \{\infty\}$; this is essentially the
  setting of Douady-Hubbard's original paper.
\end{remark}


\section{Extremal length on thickened surfaces}
\label{sec:thicken}

Our next goal is to relate extremal length on elastic graphs and on
Riemann surfaces. We first recall
different definitions of extremal length on surfaces.

\begin{definition}
  Let $A$ be a conformal annulus. Then the \emph{extremal length}
  $\EL(A)$ is defined by the following equivalent
  definitions. (Equivalence is standard.)
  \begin{enumerate}
  \item There are real numbers $s,t \in \RR_{>0}$ so that $A$ is
    conformally equivalent to the quotient of
    the rectangle
    $[0,s] \times [0,t]$ by identifying $(0,y)$ with $(s,y)$ for
    $y \in [0,t]$. Then $\EL(A)$ is $s/t$, the circumference divided
    by the height.
  \item Pick a Riemannian metric~$g$ in the conformal
    class of~$A$. For $\rho \co A \to \RR_{\ge 0}$ a suitable (Borel-measurable) scaling
    function, let $\ell_\circ(\rho g)$ be the minimal length, with
    respect to the pseudo-metric $\rho g$, of any curve homotopic to
    the core of~$A$, and let $\Area(\rho g)$ be the area of~$A$ with
    respect to~$\rho
    g$. Then
    \begin{equation}
    \EL(A) = \EL_\circ(A) = \sup_\rho \frac{\ell_\circ(\rho
      g)^2}{\Area(\rho g)}.
    \end{equation}
  \item For $g$ and $\rho$ as above, let $\ell_\perp(\rho g)$ be the
    minimal length with respect to~$\rho g$ of any curve running from
    one boundary component of~$A$ to the other. Then
    \begin{align}
      \EL_\perp(A) &= \sup_\rho \frac{\ell_\perp(\rho g)^2}{\Area(\rho g)}\\
      \EL(A) &= 1/\EL_\perp(A).\label{el-perp}
    \end{align}
  \end{enumerate}
\end{definition}

\begin{definition}
  Let $S$ be a general Riemann surface, and let $(C,c)$ be a simple
  multi-curve on~$S$ (a union of non-intersecting simple closed
  curves), with components $c_i \co C_i \to S$. Then $\EL_S[c]$ is
  defined in the following equivalent ways.
  \begin{enumerate}
  \item If $g$ is a Riemannian
    metric on~$S$ in the distinguished
    conformal class, then
    \begin{equation}\label{eq:el-test-sup}
    \EL_S[c] = \sup_\rho \frac{\ell_{\rho g}[c]^2}{\Area_{\rho g}(S)},
    \end{equation}
    where again $\rho \co S \to \RR_{\ge 0}$ runs over all Borel-measurable
    scaling factors and $\ell_{\rho g}[c]$ is the minimal length
    of any multi-curve in~$[c]$ with respect to~$\rho g$.
  \item Extremal length may be defined by finding the ``fattest'' set
    of annuli
    around~$[c]$, as follows.  For $i=1,\dots,k$, let $A_i$ be a
    (topological) annulus.  Then
    \begin{equation}\label{eq:el-inf}
      \EL[c] = \inf_{\omega, f}\,\,\sum_{i=1}^k \EL_\omega(A_i),
    \end{equation}
    where the infimum runs over all conformal structures~$\omega$
    on the~$A_i$ (i.e., a choice of modulus) and
    over all embeddings $f\co \bigsqcup_i A_i \hookrightarrow S$ that
    are conformal with respect to~$\omega$ and so that
    $f$ restricted to the core curve of~$A_i$ is isotopic to~$c_i$.
  \end{enumerate}
  These two definitions are equivalent
  \cite[Proposition
  \ref*{Emb:prop:el-total-el}]{KPT15:EmbeddingEL}.

  Also define  $\EL_\perp(\bigsqcup A_i)$, the perpendicular extremal
  length of a union of annuli, as the extremal length of the union of
  path families running between boundary components:
  \begin{equation*}
    \EL_\perp\Bigl(\bigsqcup\nolimits_i A_i\Bigr) \coloneqq
      \sup_\rho \frac{\bigl(\min_i \ell_\perp(A_i; \rho g)\bigr)^2}{\Area(\rho g)}.
  \end{equation*}
  With this definition, it is easy to verify that
  $\sum_i \EL(A_i) = 1/\EL_\perp\bigl(\bigsqcup_i A_i\bigr)$.%
  \footnote{Conceptually, in the standard definition of extremal
    length with families of paths, there are two ways to combine path
    families. You can take the \emph{union} of the two families, as in
    the definition of $\EL_\perp(\sqcup A_i)$; this decreases the
    extremal length. Alternately, you can \emph{join} the two
    families, where a valid path consists of one from each family, as
    implicitly happens in $\sum \EL(A_i)$.}
  
  For a non-simple homotopy class of multi-curves on~$S$, use
  Equation~\eqref{eq:el-test-sup} as the definition of extremal
  length.
\end{definition}

To prove Theorem~\ref{thm:embedding-thickening}, we need to estimate
extremal length on $N_\eps G$ from below and above. We prove
two propositions
for the two directions.

\begin{proposition}\label{prop:el-graph-surface}
  Let $G$ be an elastic ribbon graph. Then for $[c]$ any homotopy
  class of
  multi-curve on~$NG$, we have
  \[
  \EL_G[c] \le \eps \EL_{N_\eps G}[c].
  \]
\end{proposition}

(We use $[c]$ for the homotopy class on both $G$ and on $N_\eps G$.)

\begin{proof}
  We use Equation~\eqref{eq:el-test-sup} to estimate
  $\EL_{N_\eps G}[c]$ from below. Take as the base metric~$g$
  the standard
  piecewise-Euclidean metric in which an edge~$e$ gives an $\alpha(e)
  \times \eps$ rectangle $N_\eps e$.

  The test function~$\rho$ is the piecewise-constant function
  which is
  $n_c(e)$ on $N_\eps e$. (Recall that $n_c(e)$ is the number of times
  $c$ runs over~$e$.) Then the shortest representative of~$[c]$ will
  run down the center of each rectangle, so
  \begin{align*}
    \ell_{\rho g}[c] &= \sum_{e \in \Edges(\Gamma)} (n_c(e))^2 \alpha(e).\\
  \intertext{On the other hand, the area is}
    \Area(\rho g) &=
      \sum_{e \in \Edges(G)} (\alpha(e) n_c(e)) \cdot (\eps n_c(e)),
  \end{align*}
  so
  \[
    \eps\EL_{N_\eps G}[c] \ge
    \frac{\eps\ell_{\rho g}[c]^2}{\Area(\rho g)} =
      \frac{\eps\Bigl(\sum n_c(e)^2 \alpha(e)\Bigr)^2}{\eps\sum n_c(e)^2 \alpha(e)}
      = \EL_G[c]. \qedhere
  \]
\end{proof}

\begin{proposition}\label{prop:el-surface-graph}
  Let $G = (\Gamma,\alpha)$ be an elastic ribbon graph with trivalent
  vertices, and let
  $m = \min \{\,\alpha(e)\mid e \in \Edges(\Gamma)\,\}$ be the lowest
  weight of any edge in~$G$.  Then, for $\eps < m/2$ and $c$ any
  simple multi-curve on~$NG$, we have
  \[
  \eps \EL_{N_\eps G}[c] < \EL_G[c]\cdot (1 + 8\eps/m).
  \]
\end{proposition}

\begin{proof}
  We use Equation~\ref{eq:el-inf} to estimate $\EL_{N_\eps G}[c]$ from
  above.

  First find suitable embedded annuli. We have
  $n_c(e)$ sections of annuli running over $N_\eps e$, which we
  divide into pieces
  corresponding to these different annuli. Divide up the central
  portion of $N_\eps e$ into $n$ horizontal strips of equal
  height $\eps/n$. Inside an $\eps \times \eps$ square near each end, make
  adjustments so the annuli will glue together well at the
  vertices. (These squares do
  not overlap
  since $\eps < m/2$.) Specifically, near one end of~$e$, $n_1$ of the
  annulus sections will continue to the left-hand neighbor of~$e$ at
  the corresponding vertex and $n_2$ will continue to the right-hand
  neighbor, with $n_1 + n_2 = n_c(e)$. Divide the interval $[0,\eps/2]$ into
  $n_1$ equal sections, divide the interval $[\eps/2,\eps]$ into $n_2$ equal
  sections, and connect the corresponding endpoints by diagonal
  lines. Do the same near the other end of~$e$. Let $A$ be the
  resulting union of conformal annuli, as shown in
  Figure~\ref{fig:annuli-estimate}.
  \begin{figure}
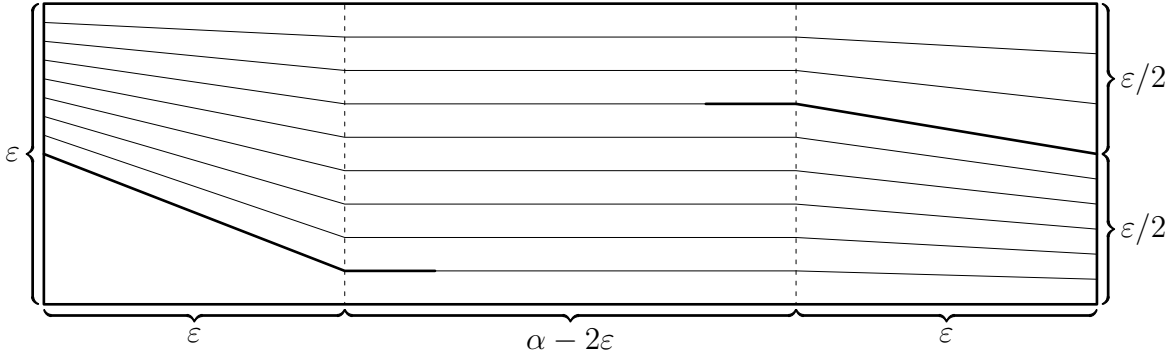

    \[
    \mfigb{el-estimate-0}
    \]
    \caption{The annuli for
      Proposition~\ref{prop:el-surface-graph} inside the rectangle
      corresponding to an edge of elastic
      weight~$\alpha$. Each portion of an
      annulus lies within a strip bounded by horizontal segments
      in the middle and diagonal segments near the ends.}
    \label{fig:annuli-estimate}
  \end{figure}

  To give an upper bound on the total extremal length
  of annuli in~$A$, we will give a lower bound on
  $\EL_\perp(A)$.  We do this by considering the restriction to the
  annuli of a suitable test
  metric~$\rho g$, where $g$ the standard
  piecewise\hyp Euclidean metric from the proof of
  Proposition~\ref{prop:el-graph-surface}.

  With this setup, take $\rho$ to be $n_c(e)$ on the central section
  of $N_\eps e$ and $\sqrt{5} n_c(e)$ on the squares at the ends
  of $N_\eps e$. In the standard metric~$g$, the vertical height of
  the annuli is $\eps/n_c$ in the center section and at least $\eps/2n_c$ in the
  end squares. In the metric $\rho g$, the vertical height is
  at least $\eps$ in the center section and
  $\eps \cdot \sqrt{5}/2$ in the squares. In the squares, since the edges
  of the annuli are sloped, the actual distance $\ell_\perp(A)$
  between the boundaries may be less
  than the vertical height; but since the slope of the edges dividing
  different pieces of~$A$ is in
  $[-1/2, 1/2]$, we have
  \[
  \ell_{\rho g,\perp}(A) \ge \eps \cdot \sqrt{5}/2 \cdot
    \cos(\tan^{-1}(1/2)) = \eps.
  \]
  
  Thus we have
  \begin{align*}
    \Area(\rho g)
       &= \sum_{e \in \Edges(\Gamma)}
                    \bigl[n_c(e)^2 (\alpha(e)-2\eps) \eps +
                     2\cdot 5 n_c(e)^2 \eps^2\bigr] \\
       &\le \eps\EL_G[c]\cdot (1 + 8\eps/m)\\
    \EL_{\perp}(A) &\ge \frac{\ell_{\rho g, \perp}(A)^2}{\Area(\rho g)}
                     \ge \frac{\eps}{\EL_G[c]\cdot (1 + 8\eps/m)}\\
    \eps\EL_{N_\eps G}[c] &\le \eps/\EL_{\perp}(A) \le \EL_G[c]\cdot (1 + 8\eps/m).
  \end{align*}
  The test function~$\rho$ is
  never continuous, so there is a strict inequality.
\end{proof}

The constant in Proposition~\ref{prop:el-surface-graph}
depends only on the local geometry of~$G$ and is thus unchanged under
covers.

\begin{remark}
  The restriction to trivalent graphs in
  Proposition~\ref{prop:el-surface-graph} can presumably be removed.
  Since every graph is homotopy-equivalent to a
  trivalent graph, it is not necessary for our applications.
\end{remark}

We have to do a little more work to deduce
Theorem~\ref{thm:embedding-thickening} from
Propositions~\ref{prop:el-graph-surface}
and~\ref{prop:el-surface-graph}: stretch factor for graphs is defined
with respect to \emph{all} multi-curves, while for surfaces we restrict
attention to \emph{simple} multi-curves. We must check that the difference
between two notions of stretch factor does not matter.

\begin{definition}
  Let $S_1$ and $S_2$ be Riemann surfaces. For $\phi \co S_1
  \hookrightarrow S_2$ a topological embedding, the \emph{simple stretch
  factor} $\SFsimp[\phi]$ is the
  stretch factor from Definition~\ref{def:sf-surf}.
  For $\phi \co S_1 \to S_2$ a continuous map, the \emph{general
    stretch factor} is
  \[
  \SFgen[\phi] \coloneqq
    \sup_{[c] \co C \to S_1} \frac{\EL[\phi \circ c]}{\EL[c]}
  \]
  where the supremum runs over \emph{all} multi-curves
  on~$S_1$ (not necessarily simple).

  Let $G_1$ and $G_2$ be elastic graphs. For $\phi \co G_1 \to G_2$ a
  map between them, the \emph{general stretch factor} $\SFgen[\phi]$
  is the stretch factor from Definition~\ref{def:sf-graph}.

  Now suppose $\phi \co G_1 \to G_2$ is a ribbon map between ribbon
  graphs. A \emph{simple multi-curve}
  on~$G_1$ is a multi-curve $c \co C \to G_1$ that
  lifts to a simple multi-curve on~$NG_1$ (i.e., so that $c$
  is a ribbon map). Then the \emph{simple stretch factor} is
  \[
  \SFsimp[\phi] \coloneqq
    \sup_{\text{$c$ simple on $G_1$}} \frac{\EL[\phi \circ c]}{\EL[c]}
  \]
  where the supremum runs over all homotopy classes of simple multi-curves
  $c\co C \to G_1$ on~$G_1$. Observe that if $c$ is a simple multi-curve,
  then $\phi \circ c$ is
  also a simple multi-curve, since $N\phi$ is an embedding.
\end{definition}

\begin{proposition}\label{prop:sf-gen-simp}
  Let $\phi \co G_1 \to G_2$ be a ribbon map between
  elastic ribbon graphs. Then
  \[
  \SFgen[\phi] = \SFsimp[\phi] = \Emb[\phi].
  \]
\end{proposition}

To prove Proposition~\ref{prop:sf-gen-simp}, we use train tracks.

\begin{definition}[{\cite[Definition \ref*{Elast:def:tt}]{Thurston19:Elastic}}]
  A \emph{train track}~$T$ is a graph in which the edges incident to
  each vertex are partitioned into equivalence classes, called
  \emph{gates}, with at least two gates at each vertex.  A
  \emph{train-track (multi-)curve} on~$T$ is a (multi-)curve that enters
  and leaves
  by different gates each time it passes through a vertex. A
  \emph{weighted train track} is a train track~$T$ with a weight
  $w(e)$ for each edge~$e$ of~$T$, satisfying a triangle inequality at
  each gate~$g$ of each vertex~$v$:
  \begin{equation}\label{eq:tt-triangle}
    w(g) \le \sum_{\substack{g'\text{ gate at $v$}\\g' \ne g}} \!w(g')
  \end{equation}
  where $w(g)$ is the sum of the weights of all edges in~$g$. (If
  there are only two gates at~$v$, this inequality is necessarily an equality.)
\end{definition}

\begin{lemma}\label{lem:approx-simple}
  Let $(T,w)$ be a weighted train track with a ribbon structure. Then
  there is a sequence
  of simple train-track multi-curves $(C_i, c_i)$ and positive weighting factors $k_i$
  so that $k_i c_i$ approximates~$w$, in the sense that
  \[
  \lim_{i \to \infty} k_i n_{c_i} = w.
  \]
\end{lemma}

This lemma is close to standard facts in the theory of train
tracks. Note there is no assumption that the train track structure and
the ribbon structure are compatible. (This avoids assuming that the
optimizer for $\Emb[\phi]$ is a ribbon map.)
Compare
Lemma~\ref{lem:approx-simple} to \cite[Proposition
\ref*{Elast:prop:tt-curve}]{Thurston19:Elastic}, which gives the exact
weights (without approximation), but yields multi-curves that are not simple.

\begin{proof}
  We first prove that if $w$ is integer-valued and has even total
  weight at each vertex then there is a simple train-track multi-curve
  $(C,c)$ so that $n_c = w$. 

  On each edge~$e$ of~$T$, take $w(e)$ parallel
  strands on $Ne$. We must show how to stitch together these strands
  at the vertices without crossing strands or making illegal
  train-track
  turns. Focus on a vertex~$v$. If one of the incident edges has zero
  weight, delete it. If one of the train-track triangle inequalities
  is an equality, smooth the vertex (in the sense of
  \cite[Definition \ref*{Elast:def:tt}]{Thurston19:Elastic}) so that there are
  only two gates at~$v$.
  After this, if there are at least three gates
  at~$v$, then all inequalities are strict and Equation~\eqref{eq:tt-triangle}
  is strengthened to
  \begin{equation}
    \label{eq:triang-strict}
    w(g) \le -2 + \sum_{\substack{\text{$g'$ gate at $v$}\\g' \ne g}} \!w(g')
  \end{equation}
  by the parity condition.

  In either the two-gate or more-gate case, find any two edges at~$v$
  that are adjacent in the ribbon
  structure and belong to different gates. Join adjacent outermost
  strands from these two edges. We are left with a smaller problem,
  where the weights on these two strands are reduced by~$1$. The
  train-track inequalities are still satisfied, using
  Equation~\eqref{eq:triang-strict} when there are three or more
  gates. By induction we can join up all the strands to a
  multi-curve~$C$. Because we always join strands that are adjacent among
  the remaining strands in the ribbon structure, $C$ is
  simple.

  For general weights, find a sequence of even integer
  weights~$w_i$ on~$T$ and factors $k_i$ so that
  $\lim_{i \to \infty} k_i w_i = w$
  and the $w_i$ satisfy the train-track inequalities for~$T$. The
  above argument gives a simple multi-curve $(C_i,c_i)$ for each~$w_i$, as
  desired.
\end{proof}

\begin{proof}[Proof of Proposition~\ref{prop:sf-gen-simp}]
  We already know that $\Emb[\phi] = \SFgen[\phi]$. From the
  definition, it is clear that $\SFsimp[\phi] \le
  \SFgen[\phi]$. It remains to prove that $\Emb[\phi] \le
  \SFsimp[\phi]$.
  By \cite[Proposition
  \ref*{Elast:prop:emb-sf-detail}]{Thurston19:Elastic}, there is a
  weighted train-track~$T$ that fits into a
  tight sequence
  \[
  \shortseq{T}{t}{G_1}{\psi}{G_2}
  \]
  where $t$ is the inclusion of a subgraph, $\psi \in [\phi]$ and
  $\psi \circ t$ is a train-track
  map. (``Tight'' means that the energies are multiplicative, which in
  this case means that $\EL[\psi \circ t] = \EL[t] \Emb[\psi]$, and
  furthermore all three maps $t$, $\psi$, and $\psi \circ t$ are
  minimizers in their homotopy classes.) $T$
  inherits a ribbon structure as a subgraph of~$G_1$.
  By Lemma~\ref{lem:approx-simple}, we can find a sequence of simple
  multi-curves
  $(c_i,C_i)$ on~$T$ so that
  \[
  C_i \overset{c_i}{\longrightarrow} T
      \overset{t}{\longrightarrow} G_1
      \overset{\psi}{\longrightarrow} G_2
  \]
  approaches a tight sequence, in the sense that $t \circ c_i$ and
  $\psi \circ t \circ c_i$ are both reduced and
  \[
  \lim_{i \to \infty} \frac{\EL[\psi \circ t\circ c_i]}
       {\EL[t \circ c_i]}
     = \frac{\EL[\psi \circ t]}{\EL[t]} = \Emb[\phi].
  \]
  Since $t$ is an inclusion, the sequence of weighted multi-curves
  $t \circ c_i$ is simple, as desired.
\end{proof}

\begin{proof}[Proof of Theorem~\ref{thm:embedding-thickening}]
  Immediate from Propositions~\ref{prop:el-graph-surface},
  \ref{prop:el-surface-graph}, and~\ref{prop:sf-gen-simp}.
\end{proof}

\begin{question}
  For $\phi \co S_1 \hookrightarrow S_2$ a topological embedding of
  Riemann surfaces, how does $\SFgen[\phi]$ behave? By considering
  quadratic differentials, it is not hard to
  see that if $\phi$ is not homotopic to an annular conformal
  embedding, then
  \begin{align*}
    \SFsimp[\phi] &= \SFgen[\phi] \ge 1.\\
    \intertext{On the other hand, if $\phi$ is homotopic to a
    conformal embedding then}
    \SFsimp[\phi] &\le \SFgen[\phi] \le 1.
  \end{align*}
  But this leaves many questions open.
\end{question}


\section{Iteration and asymptotic stretch factor}
\label{sec:iter}

\subsection{General theory}
\label{sec:general-theory}
To complete the
proof of Theorem~\ref{thm:detect-rational}, we turn to the behavior
of energies under iteration.
Recall first that if $\phi \co X_1 \to X_0$ is any continuous map and
$\pi_0 \co Y_0 \to X_0$ is a covering map, we can form the pull-back
\[
\begin{tikzpicture}[x=1.5cm,y=1.5cm]
  \node (X1) at (0,0) {$X_1$};
  \node (X0) at (1,0) {$X_0.$};
  \node (Y1) at (0,1) {$Y_1$};
  \node (Y0) at (1,1) {$Y_0$};
  \draw[->] (X1) to node[above,cdlabel]{\phi} (X0);
  \draw[->] (Y0) to node[right,cdlabel]{\pi_0} (X0);
  \draw[->,dashed] (Y1) to node[above,cdlabel]{\wt\phi} (Y0);
  \draw[->,dashed] (Y1) to node[right,cdlabel]{\pi_1} (X1);
\end{tikzpicture}
\]
Then $\pi_1$ is also a covering map. In this setting,
we call $\wt{\phi}$ a \emph{cover} of~$\phi$.

\begin{definition}\label{def:correspondence}
  A \emph{topological correspondence} is a pair of topological spaces
  $X_1$ and~$X_0$ and a pair of maps between them: $\pi,\phi \co X_1
  \rightrightarrows X_0$.

  For $n \ge 0$, the $n$'th \emph{orbit space} $X_n$
  of a topological correspondence is the $n$-fold product of $X_1$
  with itself over~$X_0$ using $\pi$ and~$\phi$, i.e., the pull-back $X_n$
  in Figure~\ref{fig:pullback}.
\begin{figure}
  \[
  \begin{tikzpicture}
    \node (X00) at (0,0) {$X_0$};
    \node (X01) at (2,0) {$X_0$};
    \node (X02)at (4,0) {$\cdots$};
    \node (X03) at (6,0) {$X_0$};
    \node (X04) at (8,0) {$X_0$};

    \node (X10) at (1,1) {$X_1$};
    \node (X11) at (3,1) {$X_1$};
    \node at (4,1) {$\cdots$};
    \node (X12) at (5,1) {$X_1$};
    \node (X13) at (7,1) {$X_1$};

    \node (X3) at (4,4) {$X_n$};

    \draw[->] (X10) to node[above left=-1mm,cdlabel]{\phi} (X00);
    \draw[->] (X11) to node[above left=-1mm,cdlabel]{\phi} (X01);
    \draw (X12) to (barycentric cs:X12=1,X02=1.2);
    \draw[->] (X13) to node[above left=-1mm,cdlabel]{\phi} (X03);

    \draw[->] (X10) to node[above right=-1mm,cdlabel]{\pi} (X01);
    \draw (X11) to (barycentric cs:X11=1,X02=1.2);
    \draw[->] (X12) to node[above right=-1mm,cdlabel]{\pi} (X03);
    \draw[->] (X13) to node[above right=-1mm,cdlabel]{\pi} (X04);


    \draw[->,dashed] (X3) to (X10);
    \draw[->,dashed] (X3) to (X11);
    \draw[->,dashed] (X3) to (X12);
    \draw[->,dashed] (X3) to (X13);

    \draw[->,dashed,bend right] (X3) to node[above left=-1mm,cdlabel]{\phi^n} (X00);
    \draw[->,dashed,bend left] (X3) to node[above right=-1mm,cdlabel]{\pi^n} (X04);

    \draw[decorate,decoration={brace,amplitude=6pt}] (X04.south)--
      node[below=5pt]{$n+1$ copies of $X_0$, $n$ copies of $X_1$}
      (X00.south);
  \end{tikzpicture}
  \]
  \caption{The pullback construction of $X_n$.}\label{fig:pullback}
\end{figure}
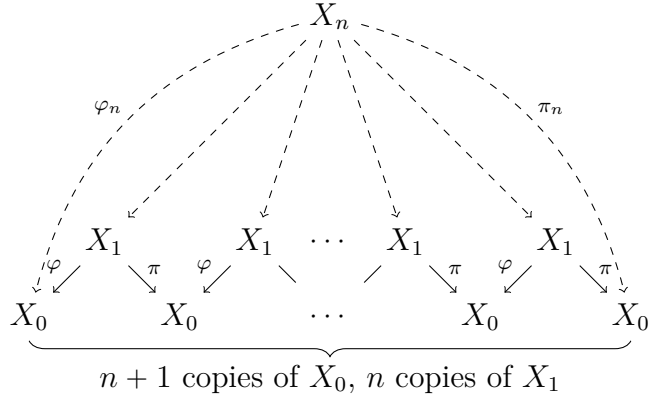
  Concretely, $X_n$ is the set of tuples
  \[
  (y_0,x_1,y_1,x_2,y_2,\dots,y_{n-1},x_n,y_n) \in X_0 \times (X_1 \times X_0)^n
  \]
  so that $\phi(x_i) = y_{i-1}$ and $\pi(x_i) = y_i$ for $1 \le i
  \le n$, with the
  subspace topology. Of the natural maps from $X_n$ to $X_0$, we distinguish
  \begin{align*}
    \phi^n(y_0,x_1,\dots,x_n,y_n) &= y_0\\
    \pi^n(y_0,x_1,\dots,x_n,y_n) &= y_n.
  \end{align*}
  The corresponding \emph{iterate} of $(\pi,\phi)$ is the virtual
  endomorphism $\pi^n,\phi^n \co X_n \rightrightarrows X_0$.
  If
  $(\pi,\phi)$ is a virtual endomorphism, so is $(\pi^n,\phi^n)$. We
  also distinguish intermediate maps $\pi^n_k, \phi^n_k \co X_n
  \rightrightarrows X_k$ for $k \le n$ by
  \begin{align*}
    \phi^n_k(y_0,x_1,\dots,x_n,y_n) &= (y_0,x_1,\dots,x_k,y_k)\\
    \pi^n_k(y_0,x_1,\dots,x_n,y_n) &= (y_{n-k},x_{n-k+1},\dots,x_n,y_n).
  \end{align*}
  (The convention is that superscripts refer to the domain and
  subscripts refer to the range. This fits well with seeing $\phi^n$
  as an ``iterate'' of $\phi$.)

  If we have two topological correspondences
  $\pi_X, \phi_X \co X_1 \rightrightarrows X_0$ and
  $\pi_Y, \phi_Y \co Y_1 \rightrightarrows Y_0$ and a morphism
  $(f_0, f_1)$ from $(\pi_X, \phi_X)$ to $(\pi_Y, \phi_Y)$ (in the
  sense of Definition~\ref{def:virt-homotopy-equiv}),
  then we
  can also use the
  pull-back property to iterate the morphism, getting a map
  $f_n \co X_n \to Y_n$. Concretely,
  \[
  f_n(y_0,x_1,y_1,\dots,x_n,y_n) = (f_0(y_0),f_1(x_1),f_0(y_1),\dots,f_1(x_n),f_0(y_n)).
  \]
\end{definition}

If $\phi$ is injective (as for surface virtual endomorphisms from
branched self-covers), $\pi \circ \phi^{-1}$ is a partially-defined
map on~$X_0$. Then $\phi^n$ is also injective and
$\pi^n\circ (\phi^n)^{-1}$, where defined, is the $n$-fold iterate
$(\pi \circ \phi^{-1})^{\circ n}$. This
justifies the term `iteration'.

\begin{definition}\label{def:energy-iterate}
  Consider a category of spaces with a structure that
  can be lifted to covers (like an elastic structure on graphs or a
  conformal structure on surfaces). Suppose that we have a non-negative
  \emph{energy} $E$ defined for suitable maps $\phi \co X \to Y$
  that is
  \emph{sub-multiplicative}, in the sense that
  \begin{equation}\label{eq:submult}
    E(\psi \circ \phi) \le E(\psi) E(\phi),
  \end{equation}
  and \emph{invariant under covers}, in the sense that if
  $\wt \phi \co \wt X \to \wt Y$
  is a cover of~$\phi$, then
  \begin{equation}\label{eq:cover-invariant}
    E(\wt \phi) = E(\phi).
  \end{equation}
  Then for $\pi, \phi \co X_1 \rightrightarrows X_0$ a virtual
  endomorphism between such spaces, where the structure
  on~$X_1$ is lifted from the structure on~$X_0$ via~$\pi$,
  the \emph{asymptotic energy} is
  \begin{equation}
    \overline{E}(\pi,\phi) \coloneqq \lim_{n \to \infty} E(\phi^n)^{1/n}.
  \end{equation}
\end{definition}

\begin{proposition}\label{prop:asympt-e}
  Let $\pi,\phi \co X_1 \rightrightarrows X_0$ be a virtual
  endomorphism and let $E$ be an energy
  that is sub-multiplicative and invariant under covers. Then
  the limit defining the asymptotic energy converges and is equal to
  the infimum of the terms. In particular,
  $\overline{E}(\pi,\phi) \le E(\phi)$.
\end{proposition}

\begin{proof}
  Note that $\phi^n = \phi^k \circ \phi^n_k$.
  An examination of the diagrams reveals that $\phi^n_k$ is a cover
  of $\phi^{n-k}$, so
  $E(\phi^n_k) = E(\phi^{n-k})$. We
  therefore have
  $E(\phi^n) \le E(\phi^k)E(\phi^{n-k})$ by Equation~\eqref{eq:submult}. Then
  the sequence $\log(E(\phi^n))$ is sub-additive, and Fekete's
  Lemma gives the result.%
\footnote{Fekete's Lemma: 
  if $(a_n)_{n=1}^\infty$ is sub-additive, then
  $\lim_{n \to \infty}a_n/n$ exists and is equal to the infimum of the
  terms.}
\end{proof}

If an energy $E(\phi)$ is invariant under homotopy of~$\phi$, we will
write it as $E[\phi]$.

\begin{proposition}\label{prop:asympt-homotopy}
  Let $\pi,\phi \co X_1 \rightrightarrows X_0$ be a virtual
  endomorphism and let $E$ be an energy that is sub-multiplicative,
  invariant under covers, and invariant under
  homotopy. Then $\overline{E}(\pi,\phi)$ is invariant under
  homotopy equivalence of $(\pi,\phi)$.
\end{proposition}

\begin{proof}
  Let $\sigma, \psi \co Y_1 \rightrightarrows Y_0$ be a virtual
  endomorphism homotopy equivalent to $(\pi,\phi)$, with homotopy
  equivalences given by $f_i \co X_i \to Y_i$ and
  $g_i \co Y_i \to X_i$ for $i=0,1$. We need to compare $E[\phi^n]$ and
  $E[\psi^n]$.
  Let $f_n \co X_n \to Y_n$ be the iterate of $(f_0,f_1)$. Since $f_n$ is a cover of $f_0$, we
  have $E[f_n] = E[f_0]$. Furthermore, $\phi^n$ is
  homotopic to $g_0 \circ \psi^n \circ f_n$. Then
  \begin{align*}
    E[\phi^n] &\le E[g_0] E[\psi^n] E[f_n]
                = E[\psi^n] \bigl(E[f_0] E[g_0]\bigr)\\
    E[\phi^n]^{1/n} &\le E[\psi^n]^{1/n} \bigl(E[f_0]E[g_0]\bigr)^{1/n}.\\
    \intertext{Passing to the limit on both sides, we have}
    \overline{E}(\pi,\phi) & \le \overline{E}(\sigma,\psi).
  \end{align*}
  By the same reasoning in the other direction,
  $\overline{E}(\sigma,\psi) \le \overline{E}(\pi,\phi)$.
\end{proof}

If
$E$ is sub-multiplicative, invariant under covers, and
invariant under homotopy, we will write $\oE[\pi,\phi]$ to indicate that the
asymptotic energy is independent of homotopy equivalence.

\subsection{Specific energies}
\label{sec:specific-energies}

We now turn to the specific energies of interest on elastic graphs or
conformal surfaces. There are three
relevant energies:
\begin{itemize}
\item the stretch factor $\SF$ for conformal surfaces;
\item the stretch factor $\SF$ for elastic graphs; and
\item the embedding energy $\Emb$ for elastic graphs.
\end{itemize}
The last two are equal by Theorem~\ref{thm:emb-sf}, although we
sometimes distinguish when we need to use that theorem.
All three are invariant under homotopy by definition.

For $\phi$ a map between elastic graphs, the embedding
energy $\Emb[\phi]$ is sub-multiplicative by \cite[Proposition
\ref*{Elast:prop:sub-mult}]{Thurston19:Elastic}.
\begin{lemma}
  Let $\phi \co X \to Y$ and $\psi \co Y \to Z$ be either topological
  embeddings of conformal surfaces or maps between elastic
  graphs. Then stretch factor is sub-multiplicative:
  \[
  \SF[\psi \circ \phi] \le \SF[\phi] \SF[\psi].
  \]
\end{lemma}

\begin{proof}
  In either case, the stretch factor is a supremum over
  multi-curves (simple multi-curves for maps between surfaces). For
  $c$ any suitable multi-curve on~$X$, if $\EL_Y[\phi\circ c] \ne 0$ we have
  \[
  \frac{\EL_Z[\psi \circ \phi \circ c]}{\EL_X[c]}
    = \frac{\EL_Z[\psi\circ\phi\circ c]}{\EL_Y[\phi\circ c]}
    \frac{\EL_Y[\phi \circ c]}{\EL_X[c]}
    \le \SF[\psi] \SF[\phi].
  \]
  If $\EL_Y[\phi \circ c] = 0$, then also
  $\EL_Z[\psi\circ\phi\circ c] = 0$ and we get the same inequality. Since
  $\SF[\psi \circ \phi]$ is the supremum of the left-hand side over all $c$, we get the
  desired result.
\end{proof}

\begin{proposition}\label{prop:sf-cover-graph}
  Stretch factor for maps between elastic graphs is invariant under covers.
\end{proposition}

\begin{proof}
  Let $\phi \co G \to H$ be a map between elastic graphs and let
  $\wt \phi \co \wt G \to \wt H$ be a cover of~$\phi$ of degree~$d$. First note
  that we can pull-back a multi-curve $(C,c)$ on~$G$ to a multi-curve $(\wt C,\wt c)$
  on~$\wt G$, with $\EL[\wt c] = d \EL[c]$ and
  $\EL[\wt \phi \circ \wt c] = d \EL[\phi \circ c]$. It follows that
  $\SF[\wt \phi] \ge \SF[\phi]$.

  For the other inequality, we use Theorem~\ref{thm:emb-sf}. Let
  $\psi \in [\phi]$ be a map with $\Emb(\psi) = \Emb[\phi]$, and let
  $\wt \psi$ be the corresponding lift. Then
  $\Emb[\wt \phi] \le \Emb(\wt \psi) = \Emb(\psi) = \Emb[\phi]$.
\end{proof}

For graphs, embedding energy/stretch factor fits nicely into
the general
theory laid out in Section~\ref{sec:general-theory}.
For
surfaces, $\SF$ is not invariant under covers
\cite[Example~\ref*{Emb:examp:cover}]{KPT15:EmbeddingEL}.
We therefore modify the definition.

\begin{definition}
  For $\phi \co R \hookrightarrow S$ a topological
  embedding of conformal surfaces, the \emph{lifted stretch factor}
  $\wt{\SF}[\phi]$ is
  \[
  \wt{\SF}[\phi] \coloneqq \sup_{\substack{\text{$\wt \phi$ finite}\\\text{cover of $\phi$}}} \SF[\wt \phi].
  \]
\end{definition}

\begin{citethm}[Kahn-Pilgrim-Thurston {\cite[Theorem \ref*{Emb:thm:sf-cover}]{KPT15:EmbeddingEL}}]
  \label{thm:sf-cover}
  Let $\phi \co R \hookrightarrow S$ be a topological
  embedding of Riemann surfaces. If $\SF[\phi] \ge 1$, then
  $\wt\SF[\phi] = \SF[\phi]$. If $\SF[\phi] < 1$, then
  \[
  \SF[\phi] \le \wt\SF[\phi] < 1.
  \]
\end{citethm}

\begin{lemma}
  $\wt \SF$ is sub-multiplicative.
\end{lemma}

\begin{proof}
  Any cover of a composition $\phi \circ \psi$ factors as a
  composition $\wt \phi \circ \wt \psi$ of covers of $\phi$ and~$\psi$.
\end{proof}

\begin{lemma}
  $\wt \SF$ is invariant under covers.
\end{lemma}

\begin{proof}
  Any two finite covers of a
  map have a common finite cover.
\end{proof}

Let us recap what we have so far. For
$\pi_G,\phi_G \co G_1 \rightrightarrows G_0$ an virtual
endomorphism of elastic graphs, we have an asymptotic energy
$\ASF[\pi_G,\phi_G]$, invariant under homotopy equivalence. In
particular, $\ASF$ is independent of the elastic structure on~$G_0$,
since the identity is a homotopy equivalence..

For $\pi_S, \phi_S \co S_1 \rightrightarrows S_0$ a virtual
endomorphism of conformal surfaces, we have an asymptotic energy
$\overline{\wt{\SF}}[\pi_S,\phi_S]$, which we will also write
$\ASF[\pi_S,\phi_S]$. (See Corollary~\ref{cor:asf-lim} below.) This is
invariant under quasi-conformal
homotopy equivalences and therefore is independent of the conformal
structure, as long as we don't change a puncture to a boundary
component or vice versa.

In particular, if $\pi, \phi \co G_1 \rightrightarrows G_0$ is a
ribbon virtual endomorphism of elastic ribbon graphs, then 
the asymptotic energy of the induced virtual endomorphism
$N_\eps\pi, N_\eps\phi \co N_\eps G_1 \rightrightarrows N_\eps G_0$ is
independent of~$\eps$. Thus we will drop $\eps$ from the notation.

\begin{proposition}\label{prop:asf-graph-surface}
  Let $\pi,\phi \co \Gamma_1 \rightrightarrows \Gamma_0$ be a ribbon graph
  virtual endomorphism. Then
  \[
  \ASF[\pi,\phi] = \ASF[N\pi,N\phi] =
    \lim_{n \to \infty} \sqrt[n]{\SF[N\phi^n]}.
  \]
\end{proposition}

\begin{proof}
  By Proposition~\ref{prop:asympt-homotopy}, we can replace the
  virtual endomorphism with a ribbon homotopy equivalent one without
  changing $\ASF$. So we may assume that $G_0$
  and~$G_1$ are
  trivalent, with some minimum elasticity~$m$ on any edge. Pick
  $\eps < m/2$.

  Theorem~\ref{thm:embedding-thickening} says that $\SF[N_\eps\phi^n]$ is
  within a factor of $1+8\eps/m$ of $\SF[\phi^n]$ for
  all~$n$. Similarly, since $\SF$ for graphs is invariant under covers
  and the estimates depend only on the local geometry,
  $\wt{\SF}[N_\eps\phi^n]$ is within a factor of
  $1 + 8\eps/m$ of $\wt\SF[\phi^n]$. When we take the
  $n$'th root in limit for
  the three terms in the statement, this factor disappears, as
  in Proposition~\ref{prop:asympt-homotopy}.
\end{proof}

\begin{corollary}\label{cor:asf-lim}
  For any virtual surface endomorphism
  $\pi,\phi \co S_1 \rightrightarrows S_0$ where $S_0$ and~$S_1$ have
  no punctures,
  \[
  \ASF[\pi,\phi] = \lim_{n \to \infty} \sqrt[n]{\SF[\phi^n]}.
  \]
\end{corollary}

\begin{proof}
  If $S_0$ and $S_1$ are closed surfaces, then $\SF[\phi] \ge 1$
  (since there is never an annular conformal embedding) so $\SF$
  is invariant under covers and the statement is trivial.

  Otherwise, we will see that $\lim \sqrt[n]{\SF[\phi^n]}$ is
  independent of the conformal structure on~$S_0$ (since the general
  theory of asymptotic energies does not apply), and then replace
  $S_0$ by an $\epsilon$-thickening of a graph and apply
  Proposition~\ref{prop:asf-graph-surface}.

  For that purpose, let $\pi,\phi \colon S_1 \rightrightarrows S_0$
  and $\pi',\phi'\colon S_1' \rightrightarrows S_0'$ be two virtual
  endomorphisms of conformal surfaces, and let $f_i \colon S_i \to
  S_i'$ be one direction of a homotopy equivalence between them. If
  $\SF[f_0] \ge 1$, then $\SF[f_n] = \SF[f_0]$ for all $n$ by
  the first part of Theorem~\ref{thm:sf-cover}. If $\SF[f_0] < 1$,
  then $\SF[f_n] < 1$ by the second part of
  Theorem~\ref{thm:sf-cover}. Either way, we have the inequalities we
  need to deduce that $\SF[\phi^n]$ is within a uniform constant factor of
  $\SF[\phi'{}^n]$, as in the proof of Proposition~\ref{prop:asympt-homotopy}.
\end{proof}

\subsection{Proof of Theorem~\ref{thm:detect-rational}}
\label{sec:proof-theorem}

We are now ready to prove Theorem~\ref{thm:detect-rational}. We expand
the statement to include asymptotic energies.

\begin{taggedthm}{1\/$'$}\label{thm:detect-rational-a}
  Let $f \co (\Sigma,P) \righttoleftarrow$ be a branched
  self-cover of hyperbolic type with associated surface virtual
  endomorphism $(\pi_S, \phi_S)$. Then the following
  conditions are equivalent.
  \begin{enumerate}
  \item\label{item:rational} $f$ is equivalent to a rational map;
  \item\label{item:exists-emb} there is an elastic graph spine~$G$ for
    $\Sigma\setminus P$ and an integer $n > 0$ so that
    $\Emb[\phi_G^n] < 1$;
  \item\label{item:every-emb} for every elastic graph spine~$G$ for
    $\Sigma\setminus P$ and for every sufficiently large~$n$
    (depending on $f$ and~$G$), we have $\Emb[\phi_G^n] < 1$;
  \item\label{item:ASF-surf} $\ASF[\pi_S, \phi_S] < 1$; and
  \item\label{item:ASF-graph} $\ASF[\pi_G, \phi_G] < 1$.
  \end{enumerate}
\end{taggedthm}

\begin{proof}
  Conditions~(\ref{item:ASF-surf}) and~(\ref{item:ASF-graph}) are
  equivalent by
  Proposition~\ref{prop:asf-graph-surface}. Conditions~(\ref{item:exists-emb})
  and~(\ref{item:every-emb}) are equivalent to
  Condition~(\ref{item:ASF-graph}) by
  Proposition~\ref{prop:asympt-e} and Theorem~\ref{thm:emb-sf}.

  To see that Condition~(\ref{item:rational}) implies
  Condition~(\ref{item:ASF-surf}), suppose $f$ is equivalent to a rational map. Then by
  Theorem~\ref{thm:rational-surfaces-embed}, there is a conformal
  virtual endomorphism
  $\pi_S, \phi_S \co S_1 \rightrightarrows S_0$ compatible
  with~$f$. Since $\phi_S$ is annular, by Theorem~\ref{thm:sf-cover}
  $\wt\SF[\phi_S] < 1$, so by Proposition~\ref{prop:asympt-e},
  $\ASF[\pi_S, \phi_S] < 1$.

  Conversely, to see Condition~(\ref{item:ASF-surf}) implies
  Condition~(\ref{item:rational}), suppose $\ASF[\pi_S,\phi_S] < 1$ with respect to any
  conformal structure~$S_0$. By
  Proposition~\ref{prop:asympt-e}, there is some~$n$ so that
  $\SF[\phi_S^n] \le \wt\SF[\phi_S^n] < 1$, so by Theorem~\ref{thm:emb-surf},
  $\phi_S^n$ is homotopic to an annular conformal embedding. Then by
  Theorem~\ref{thm:rational-surfaces-embed}, the $n$-fold composition
  $f^{\circ n}$ is equivalent to a rational map, which implies that $f$
  itself is equivalent to a rational map.
\end{proof}

\begin{remark}
  The last step in the proof follows from
  W.\ Thurston's Obstruction Theorem
  (Theorem~\ref{thm:thurston-obstruction} below), since an
  obstruction for~$f$
  is also an obstruction for~$f^{\circ n}$, but doesn't use the full
  strength of that theorem. It suffices, for instance, to know that
  some power of the
  pull-back map on Teichmüller space is contracting \cite[Section
  2.5]{BCT14:TeichHolo}.
\end{remark}


\section{Asymptotics of other energies}
\label{sec:other-energies}

The theory of asymptotic energies developed at the beginning of
Section~\ref{sec:iter} applies to any energy that is
sub-multiplicative and invariant under homotopy and
covers. In particular, it applies to any of the $p$-conformal energies
$E^p_p$ defined in \cite[Appendix
\ref*{Elast:sec:energies}]{Thurston19:Elastic}.
These energies $E^p_p[\phi]$ are a simultaneous generalization of best
Lipschitz constant $\Lip[\phi] = E^\infty_\infty[\phi]$,
and the embedding energy $\Emb[\phi] = \bigl(E^2_2[\phi]\bigr)^2$.

Recall that a $p$-conformal graph, for $1 < p \le \infty$, is a graph
with a $p$-length $\alpha(e)$ on each edge~$e$, which we will treat as
a metric. A $1$-conformal graph is instead a weighted graph, with a
weight $w(e)>0$ on each edge.
For $\phi \co G_1 \to G_2$ a PL map between
$p$-conformal graphs,
$E^p_p[\phi]$ is defined by
\begin{equation}\label{eq:Epp-std}
\begin{aligned}
  E^p_p(\phi) &\coloneqq
    \begin{cases}
      \displaystyle\esssup_{y \in G_2} \,\frac{n_{\phi,w}(y)}{w(y)} & p = 1\\
      \displaystyle\esssup_{y \in G_2} \,\Biggl(\sum_{x \in
              \phi^{-1}(y)} \abs{\phi'(x)}^{p-1}\Biggr)^{1/p} &
        1 < p < \infty\\
      \Lip(\phi) & p = \infty
    \end{cases}\\
  E^p_p[\phi] &\coloneqq \inf_{\psi \in [\phi]} E^p_p(\psi).
\end{aligned}
\end{equation}
(In the $p=1$ case, we count weighted preimages:
$n_{\phi,w}(y) = \sum_{x \in \phi^{-1}(y)} w(x)$.)
Like
$\Emb$, the energy $E^p_p$ is
sub-multiplicative and
invariant under covers, whether or not we take homotopy classes. For
a graph virtual endomorphism
$\pi, \phi \co \Gamma_1 \rightrightarrows \Gamma_0$, we thus have asymptotic energies
$\oE^p_p[\pi,\phi]$. By
Proposition~\ref{prop:asympt-homotopy}, $\oE^p_p$ is invariant under
homotopy equivalence.

$E^p_p[\phi]$ can be characterized as a stretch
factor for energies of maps to length graphs \cite[Theorem
\ref*{Elast:thm:energy-sf}]{Thurston19:Elastic}. For
$1 \le p \le \infty$ and $f \co G \to K$ a PL
map from a $p$-conformal graph to a length graph, there are energies
\begin{align}
  E^p_\infty(f) &\coloneqq \norm{f'}_p\label{eq:Epinf}\\
  \nonumber E^p_\infty[f] &\coloneqq \inf_{g \in [f]} E^p_\infty(g).
\end{align}
Then, for $\phi \co G_1 \to G_2$ a PL map between $p$-conformal graphs,
\begin{equation}
  E^p_p[\phi] = \sup_{[f] \co G_2 \to K} \frac{E^p_\infty[f \circ \phi]}{E^p_\infty[f]}
\label{eq:Epp-SF}
\end{equation}
where the supremum runs over all length graphs~$K$ and homotopy
classes $[f]$ of maps.

To study the behavior of $\oE^p_p$ as $p$ varies,
we compare $p$-conformal energies and $q$-conformal energies
for $p \le q$.
\begin{definition}
  For a metric graph~$G$, let
  \begin{align*}
    m(G) &\coloneqq \min_{e\in\Edges(G)} \alpha(e)\\
    M(G) &\coloneqq \sum_{e \in \Edges(G)} \alpha(e).
  \end{align*}
\end{definition}

\begin{lemma}\label{lem:qp-ineq}
  For $1 \le p \le q \le \infty$ and $f \co G \to K$ a
  constant-derivative map from a metric graph to a length graph,
  \[
    E^q_\infty(f) \le m(G)^{-\frac{1}{p}+\frac{1}{q}}E^p_\infty(f).
  \]
  The same inequality is true when we minimize over the homotopy class.
\end{lemma}

\begin{proof}
  In general, there is an inequality
  \begin{align}
    \label{eq:qp-ineq}\Bigl( \sum_i x_i{}^q \Bigr)^{1/q} \le
     \Bigl( \sum_i x_i{}^p \Bigr)^{1/p}.\\
    \intertext{With positive weights $w_i$ with $m = \min_i w_i$, this becomes}
    \label{eq:qp-w-ineq}\Bigl( \sum_i w_i x_i{}^q \Bigr)^{1/q} \le
     m^{-\frac{1}{p}+\frac{1}{q}} \cdot \Bigl( \sum_i w_i x_i{}^p \Bigr)^{1/p}.
  \end{align}
  (Apply Equation~\eqref{eq:qp-ineq} to the sequence
  $w_i{}^{1/q}x_i$.) Apply Equation~\eqref{eq:qp-w-ineq} to the
  definition of $E^p_\infty$.

  To get the statement for homotopy classes, apply this inequality to
  a map~$f$ that minimizes $E^p_\infty(f)$ in its homotopy class.
\end{proof}

\begin{lemma}\label{lem:pq-ineq}
  For $1 \le p \le q \le \infty$ and $f \co G \to K$ a PL
  map from a metric graph to a length graph,
  \[
    E^p_\infty(f) \le M(G)^{\frac{1}{p}-\frac{1}{q}}\cdot E^q_\infty(f).
  \]
  The same inequality is true when we minimize over the homotopy class.
\end{lemma}

\begin{proof}
  Use sub-multiplicativity of the energies \cite[Proposition \ref*{Elast:prop:energy-submult}]{Thurston19:Elastic}:
  \begin{equation*}
    E^p_\infty(f) \le E^p_q(\id) E^q_\infty(f).
  \end{equation*}
  From the definition of~$E^p_q$ \cite[Equation
  (\ref*{Elast:eq:Epq})]{Thurston19:Elastic}, we see
  that $E^p_q(\id) = M(G)^{\frac{q-p}{pq}}$. (Alternatively, apply
  Hölder's inequality to Equation~\eqref{eq:Epinf}.)

  To get the statement for homotopy classes, apply this inequality to
  a map~$f$ that minimizes $E^q_\infty(f)$ in its homotopy class.
\end{proof}

We can now see that these energies give nothing new for
ordinary
(non-virtual) graph endomorphisms (i.e., outer automorphisms of the
free group). Define the asymptotic energy of an ordinary
endomorphism~$\phi$ to be
$\oE[\phi] \coloneqq \oE[\id,\phi] = \lim_{n \to \infty}
\sqrt[n]{E[\phi^{\circ n}]}$.
\begin{proposition}\label{prop:Epp-endo}
  For $[\phi] \co \Gamma \to \Gamma$ an endomorphism of a graph,
  \[
  \oE^p_p[\phi] = \ALip[\phi].
  \]
\end{proposition}

\begin{proof}
  By Lemmas~\ref{lem:qp-ineq}
  and~\ref{lem:pq-ineq} there is a constant $C \ge 1$ so that
  \[
  \frac{1}{C} \frac{\Lip[f \circ \phi^{\circ n}]}{\Lip[f]}
  \le \frac{E^p_\infty[f \circ \phi^{\circ n}]}{E^p_\infty[f]}
  \le C \frac{\Lip[f \circ \phi^{\circ n}]}{\Lip[f]}.
  \]
  Equation~\eqref{eq:Epp-SF} then shows that $E^p_p[\phi^{\circ n}]$
  is within a factor of~$C$ of $\Lip[\phi^{\circ n}]$. The constant
  factor disappears in the limit defining $\oE^p_p[\phi]$.
\end{proof}

For virtual endomorphisms, the situation is more interesting.

\begin{proposition}\label{prop:Epp-decrease}
  For $\pi, \phi \co \Gamma_1 \rightrightarrows \Gamma_0$ a virtual
  endomorphism of graphs, $\oE^p_p[\pi,\phi]$ is a non-increasing
  function of~$p$: if $1 \le p \le q \le \infty$,
  \[
  \oE^q_q[\pi,\phi] \le \oE^p_p[\pi,\phi].
  \]
\end{proposition}

\begin{proof}
  Pick a metric structure~$G_0$ on~$\Gamma_0$ and lift it to get a
  series of metric graphs $G_n$ as usual. Then for any $f \co G_0 \to
  K$ a map to a length graph,
  \[
  \frac{E^q_\infty[f \circ \phi^n]}{E^q_\infty[f]}
    \le \frac{1}{C}\frac{E^p_\infty[f \circ \phi^n]}{E^p_\infty[f]}
  \]
  for some constant~$C$, since
  $m(G_n) = m(G_0)$.
  Now $E^q_q[\phi^n] \le E^p_p[\phi^n]/C$, and the constant factor
  disappears in the limit as usual.
\end{proof}

\begin{proposition}\label{prop:Epp-lower-bound}
  For $\pi, \phi \co \Gamma_1 \rightrightarrows \Gamma_0$ a virtual
  endomorphism of graphs with $\pi$ a covering of degree~$d$, if
  $1 \le p \le q \le \infty$,
  \[
  \oE^q_q[\pi,\phi] \ge d^{-\frac{1}{p}+\frac{1}{q}} \cdot\oE^p_p[\pi,\phi].
  \]
\end{proposition}

\begin{proof}
  Pick a metric structure $G_0$ on $\Gamma_0$ and lift it as in
  Proposition~\ref{prop:Epp-decrease}.
  Observe that $M(G_n) = d^n M(G_0)$. Then by Lemmas~\ref{lem:qp-ineq}
  and \ref{lem:pq-ineq} for any length graph~$K$ and $[f] \co G_0 \to K$,
  \[
  \frac{E^p_\infty[f \circ \phi^n]}{E^p_\infty[f]} \le
    \left(\frac{M(G_n)}{m(G_0)}\right)^{\frac{1}{p} - \frac{1}{q}}
      \frac{E^q_\infty[f \circ \phi^n]}{E^q_\infty[f]}
      = C\cdot d^{n \left(\frac{1}{p}-\frac{1}{q}\right)}\cdot E^q_q[\phi^n]
  \]
  for some constant~$C$. Then taking the supremum over~$K$ and~$[f]$,
  taking the $n$'th root, and passing to the limit shows that
  $\oE^p_p[\pi,\phi] \le d^{1/p-1/q}\cdot\oE^q_q[\pi,\phi]$.
\end{proof}

\begin{corollary}
  $\oE^p_p[\pi,\phi]$ is a continuous function of~$p$.
\end{corollary}

\begin{question}
  What more can be said about $\oE^p_p[\pi,\phi]$ as
  $p$ varies? For instance, an examination of Equation~\eqref{eq:Epp-std}
  shows that for any map $\phi \co G_1 \to G_2$ between
  metric graphs and $1 \le p \le q \le \infty$,
  \begin{align}
  \bigl(E^q_q(\phi)\bigr)^{q/(q-1)} &\le
    \bigl(E^p_p(\phi)\bigr)^{p/(p-1)}\nonumber\\
  \intertext{and so for a virtual endomorphism}
  \bigl(\oE^q_q[\pi,\phi]\bigr)^{q/(q-1)} &\le
    \bigl(\oE^p_p[\pi,\phi]\bigr)^{p/(p-1)}.\label{eq:asympt-ineq}
  \end{align}
  When $\oE^p_p[\pi,\phi] < 1$, this is stronger
  than Proposition~\ref{prop:Epp-decrease}. Is there more? For
  instance, is $\oE^p_p[\pi,\phi]$ a convex function of~$p$ in some
  sense?
\end{question}

The cases $p = 1$, $2$, or $\infty$ of the asymptotic
energy are of particular interest.
\begin{itemize}
\item The most important case is $\oE^\infty_\infty[\pi,\phi]$, with
  $\oE^\infty_\infty < 1$ iff
  $\pi, \phi \co G_1 \rightrightarrows G_0$ is a combinatorial model
  for an \emph{expanding} dynamical system in the sense of
  Nekrashevych~\cite{Nekrashevych14:CombModel}.%
  \footnote{Nekrashevych's combinatorial models are more general,
    allowing higher-dimensional cells.}
  This notion of expanding is
  quite important. In the expanding case, there is a well-defined
  (ordinary) dynamical system on an
  inverse limit Julia set, independent of the details of the
  combinatorial model.
  Furthermore, the iterated monodromy groups of
  an expanding dynamical system are
  well-behaved~\cite{Nekrashevych05:SelfSimilar}, having, for
  instance, solvable word problem, while still allowing for many
  interesting examples (e.g., groups of intermediate growth).
\item For $p=2$, Theorem~\ref{thm:detect-rational} relates
  $\oE^2_2[\pi,\phi] < 1$ to rational maps.
\item The other natural special case is
  $p=1$. If the weights are all~$1$, $E^1_1[\phi] < 1$ implies that $\phi$
  is null-homotopic, so in non-trivial cases
  $\oE^1_1[\pi,\phi] \ge 1$.
  If $\oE^\infty_\infty[\pi,\phi] < 1$ (so there is a Julia set),
  it appears that $\oE^1_1[\pi,\phi] > 1$ when the Julia set has
  Sierpiński-carpet-like behavior, and that $\oE^1_1[\pi,\phi] = 1$
  when the Julia set has many local cut points in the sense of
  Carrasco Piaggio~\cite{CP14:ConfDim}.
\end{itemize}
In forthcoming joint work with Kevin Pilgrim \cite{PT:ConfDim}, we
will show that, if
$p^*$ is the Ahlfors regular conformal dimension of the Julia set
of a virtual endomorphism $(\pi,\phi)$, then
\[
  \oE^{p^*}_{p^*}[\pi,\phi] = 1.
\]
Combined with the bounds on how $\oE^p_p$ varies as a function of~$p$,
this allows us to give concrete bounds on the Ahlfors regular
conformal dimension.


\section{Obstructions}
\label{sec:obstructions}

\subsection{Obstructions for rational maps}
\label{sec:obstr-rat}
We now relate Theorem~\ref{thm:detect-rational} to W. Thurston's
Obstruction Theorem. With an eye to generalizations, we rephrase it in
terms of elastic
multi-curves without the assumption of
hyperbolic type.

\begin{definition}\label{def:join}
  An \emph{elastic multi-curve} $A = (C, \alpha, c)$ on a
  surface~$\Sigma$ is a multi-curve $c \co C \hookrightarrow \Sigma$,
  together with an elastic structure (i.e., metric)~$\alpha$ on~$C$. The
  \emph{support} of~$A$ is the underlying multi-curve $(C,c)$.

  There are two natural operations on elastic multi-curves. First,
  if $\pi \co \wt \Sigma \to \Sigma$ is a covering map and $A$ is an
  elastic multi-curve on~$\Sigma$, there is a multi-curve
  $\wt A = \pi^{-1}(A)$ on $\wt \Sigma$ obtained by pull-back in the
  usual way.

  Second, if $A = (C,\alpha, c)$ is an elastic multi-curve on~$\Sigma$,
  then the \emph{join} $\Join(A)$ is the elastic multi-curve obtained by
  \begin{itemize}
  \item deleting all components of~$C$ whose images are null-homotopic or
    bound a punctured disk, and
  \item replacing any components
    $(C_1,\alpha_1),\dots,(C_k,\alpha_k)$ of~$A$ whose images are parallel with a single
    component $(C_0,\alpha_0)$, with elastic length obtained by the
    harmonic sum:
    \[
    \alpha_0 = \alpha_1 \oplus \dots \oplus \alpha_k =
    \frac{1}{\frac{1}{\alpha_1} + \dots + \frac{1}{\alpha_k}}.
    \]
  \end{itemize}
  (The harmonic sum comes from the parallel law for resistors or for
  springs.)
\end{definition}

\begin{definition}
  An \emph{obstruction} for~$f$ is 
  \begin{itemize}
  \item an elastic multi-curve~$A$ on~$\Sigma \setminus P$ and
  \item a map $\psi \co A \to \Join(f^{-1}(A))$
  \end{itemize}
  so that
  \begin{itemize}
  \item $\psi$ commutes up to homotopy with the maps to
    $\Sigma \setminus P$ and
  \item $\Emb(\psi) \le 1$.
  \end{itemize}
\end{definition}

\begin{remark}
  Contrast the obstruction map~$\psi$ with the map
  $\phi \co f^{-1}(G) \to G$ in the statement of
  Theorem~\ref{thm:detect-rational}: the maps are going the opposite
  direction.
\end{remark}

\begin{citethm}[W.~Thurston, Douady-Hubbard \cite{DH93:ThurstonChar}]
  \label{thm:thurston-obstruction}
  Let $f\co (\Sigma,P)\righttoleftarrow$ be a topological branched
  self-cover so that the first return map is not a Lattés map on any
  component. Then $f$ is
  equivalent to a rational map iff there is no obstruction for~$f$.
\end{citethm}

The usual formulation of Theorem~\ref{thm:thurston-obstruction} refers
to the maximum eigenvalue of a matrix constructed out of the multi-curves
underlying~$A$. The
above formulation is equivalent by Perron-Frobenius theory, as we
spell out in Proposition~\ref{prop:obstruction-e-val}
below. Intuitively, Theorem~\ref{thm:thurston-obstruction} says that
$f$ is rational iff there is no conformal collection of annuli that
gets (weakly) wider under backwards iteration.

\subsection{Obstructions for virtual endomorphisms}
\label{sec:obstr-virt-endo}

We now turn to obstructions in the more general setting
of the asymptotic $p$-conformal energies from
Section~\ref{sec:other-energies}. We also switch to virtual
endomorphisms of topological spaces (e.g., graphs) or orbifolds. (In
the context of branched self-covers $f \co (\Sigma,P)
\righttoleftarrow$, we should consider
the orbifold of~$f$.)

\begin{definition}
  For $1 < p < \infty$ and $\alpha_1, \alpha_2 \in \RR_{> 0}$, the
  \emph{$p$-harmonic sum} of $\alpha_1$ and~$\alpha_2$ is
  \begin{align}\label{eq:p-sum}
    \alpha_1 \oplus_p \alpha_2 &\coloneqq
       \bigl((\alpha_1)^{1-p} + (\alpha_2)^{1-p}\bigr)^{1/(1-p)}.\\
  \intertext{For $p = \infty$, set}
    \nonumber
    \alpha_1 \oplus_\infty \alpha_2 &\coloneqq \min(\alpha_1,\alpha_2).
  \end{align}
\end{definition}

This definition is chosen so that the $p$-energies satisfy a parallel
law.

\begin{proposition}\label{prop:join-parallel}
  For $1 < p \le q \le \infty$, let $[\phi] \co G^p \to H^q$ be a
  homotopy class of maps
  from a $p$-conformal graph to a $q$-conformal graph. Suppose that
  $G$ has two parallel edges $e_1$ and~$e_2$
  that are mapped to homotopic paths by~$\phi$. Let $G_3$ be the
  $p$-conformal graph~$G$ with $e_1$ and $e_2$ replaced by a single
  edge~$e_3$ with $\alpha(e_3) = \alpha(e_1) \oplus_p \alpha(e_2)$, and
  let $[\phi_3] \co G_3 \to H$ be the natural homotopy class. Then
  \[
  E^p_q[\phi_3] = E^p_q[\phi].
  \]
\end{proposition}

\begin{proof}
  If $q=\infty$, then the optimal maps in~$\phi$ and in~$\phi'$ will
  be constant-derivative. The result follows by examining the energy
  and comparing the derivatives. The general statement follows from
  the $q=\infty$ case by Equation~\eqref{eq:Epp-SF}.
\end{proof}

For $p=1$, we do not have a $p$-length in the same way; instead, a
$1$-conformal graph is a weighted graph, and when joining them in
parallel we add the weights.

\begin{definition}\label{def:join-top}
  A \emph{$p$-conformal multi-curve} $A = (C, \alpha, c)$ on a
  space~$X$ is a multi-curve $c \co C \to \Gamma$, together with a
  $p$-conformal structure~$\alpha$ on~$C$. (We will sometimes allow
  punning and write
  the domain of the map~$c$ as~$A$.)
  The \emph{join}
  $\Join_p(A)$ is obtained by
  deleting components of~$C$ whose image is null-homotopic or torsion
  in~$\pi_1(X)$ and
  replacing components of~$C$ whose images are parallel (up to
  homotopy) by a single component so that
  \begin{itemize}
  \item for $p>1$, the new $p$-length is the
    $p$-harmonic sum of the constituent $p$-lengths and
  \item for $p=1$, the new weight is the (ordinary)
    sum of the constituent weights.
  \end{itemize}
  
  If $\pi \co X_1 \to X_0$ is a covering map
  and~$A$ is a $p$-conformal multi-curve
  on~$X_0$, we have the usual \emph{pull-back} $\pi^* A$, a
  $p$-conformal multi-curve on~$X_1$. If $\phi \co X_1 \to X_0$ is
  a map and $A = (C, \alpha, c)$ is a $p$-conformal
  multi-curve on~$X_1$, the \emph{push-forward} $\phi_* A$ is
  $\Join_p(C, \alpha, \phi \circ c)$.

  If $\pi, \phi \co X_1 \rightrightarrows X_0$ is
  a virtual endomorphism, a \emph{$p$-obstruction} for $(\pi,\phi)$ is
  \begin{itemize}
  \item a $p$-conformal multi-curve~$A = (C,\alpha,c)$ on $\Gamma_0$ and
  \item a map $\psi \co A \to \phi_*\pi^* A$
  \end{itemize}
  so that
  \begin{itemize}
  \item $\psi$ commutes up to homotopy with the maps to~$X_0$ and
  \item $E^p_p(\psi) \le 1$.
  \end{itemize}
\end{definition}

\begin{proposition}\label{prop:join-curves}
  If $1 \le p \le q \le \infty$, $G$ is a $q$-conformal graph, and $A$
  is a $p$-conformal multi-curve on~$G$, then
  $E^p_q[A] = E^p_q[\Join_p(A)]$.
\end{proposition}
\begin{proof}
  Parallel to Proposition~\ref{prop:join-parallel}.
\end{proof}

\begin{proposition}\label{prop:obstruction-energy}
  Let $\pi,\phi \co G_1 \rightrightarrows G_0$ be a virtual
  endomorphism of $p$-conformal graphs. If there is a $p$-obstruction
  for~$(\pi,\phi)$, then $E^p_p[\phi] \ge 1$. Likewise, if
  $f \co (\Sigma, P) \righttoleftarrow$ is a topological branched
  self-cover of hyperbolic type compatible with $(\pi,\phi)$ and there
  is a $p$-obstruction for~$f$, then $E^p_p[\phi] \ge 1$.
\end{proposition}

\begin{remark}
  In the branched self-cover case, if $f$ is not of hyperbolic type
  and we do not use orbifolds, then
  considering boundary curves shows that we always have
  $E^p_p[\phi] \ge 1$. See Section~\ref{sec:future}.
\end{remark}

\begin{proof}
  Let $(A,\psi)$ be a $p$-obstruction for$(\pi,\phi)$. We have a diagram of maps
  \begin{equation*}
    \begin{tikzpicture}[baseline]
      \matrix[row sep=0.8cm, column sep=1cm]{
        \node (A2) {$\pi^* A$}; &
        \node (G1) {$G_1$}; \\
        \node (A1) {$\phi_*\pi^*A$}; \\
        \node (A0) {$A$}; &
        \node (G0) {$G_0$,}; \\
      };
      \draw[->] (A2) -- (A1);
      \draw[->] (A0) -- node[left,cdlabel]{\psi} (A1);
      \draw[->] (G1) -- node[right,cdlabel]{\phi} (G0);
      \draw[->] (A2) -- node[below,cdlabel]{\wt c} (G1);
      \draw[->] (A1) -- (G0);
      \draw[->] (A0) -- node[below,cdlabel]{c} (G0);
    \end{tikzpicture}
  \end{equation*}
  commuting up to homotopy. Since $E^p_p$ is invariant under covers,
  $E^p_p[\wt c] = E^p_p[c]$. Proposition~\ref{prop:join-curves}
  guarantees that
  \[
  E^p_p[\phi_*\pi^*A \to G_0] = E^p_p[\wt c \circ \phi].
  \]
  Then by
  sub\hyp multiplicativity and the assumptions, we have
  \begin{equation*}
    E^p_p[c] \le E^p_p[\psi] E^p_p[\wt c] E^p_p[\phi]
             \le \cdot E^p_p[c] \cdot E^p_p[\phi],
  \end{equation*}
  so $E^p_p[\phi] \ge 1$.
  The statement for branched self-covers follows immediately.
\end{proof}

See Corollary~\ref{cor:obstruction-iterate} for a strengthening of
Proposition~\ref{prop:obstruction-energy}.

\subsection{Duality}
\label{sec:duality}
There is a notion of \emph{$p$-conductance}, dual
to $p$-length.
Recall that we
can describe electrical
networks either in terms of resistances or in terms of
conductances. Resistances add in series and change by a harmonic sum
in parallel. Conductances add in parallel and change by a harmonic sum
in series. Alternately, we can think about the relation between
extremal length and modulus of conformal annuli.

There is a similar story for general $p$-conformal graphs when
$1 < p < \infty$. Recall \cite[Definition
\ref*{Elast:def:p-rescale}]{Thurston19:Elastic} that we can think of
an edge~$e$ in a $p$-conformal graph as an equivalence class of
rectangles of length~$\ell$ and height~$w$, with the $p$-length $\alpha$
given by
\[
\alpha = \frac{\ell}{w^{1/(p-1)}}.
\]
Now consider a dual view, interchanging the role of length and
height. The \emph{$p$-conductance} is
\begin{equation*}
  \gamma \coloneqq \alpha^{1-p}
                         = \frac{w}{\ell^{p-1}}
                         = \frac{w}{\ell^{1/(\pdual-1)}}
\end{equation*}
where $\pdual = p/(p-1)$ is the Hölder conjugate of~$p$.

Propositions~\ref{prop:join-parallel} and~\ref{prop:join-curves} say
that $p$-conductances add in parallel.

In checking whether a $p$-conformal multi-curve~$A$ is a $p$-obstruction,
there are two basic operations. Let us recall what
happens to the $p$-lengths.
\begin{itemize}
\item We pass to a cover by taking $f^{-1}(A)$ (i.e., the pullback by
  the covering map). A connected cover of a circle is necessarily a
  (longer) circle, which is series composition. Thus if a component of
  $f^{-1}(A)$ covers a component of $A$ by a degree~$d$ map, the
  $p$-length gets multiplied by~$d$.
\item We merge parallel components in the $\Join_p$
  operation. The $p$-weights change by the $p$-harmonic sum
  (Equation~\eqref{eq:p-sum}).
\end{itemize}
If we work with the dual $p$-conductances instead, the two operations
switch in complexity.
\begin{itemize}
\item If a circle of $p$-conductance~$\gamma$ is
  covered by a degree~$d$ map, the pull-back $p$-conductance is
  $d^{1-p}\gamma$.
\item The parallel composition in $\Join_p$ becomes simpler:
  add the constituent $p$-conductances.
\end{itemize}

\subsection{Obstruction matrices and invariant multi-curves}
\label{sec:obstruction-matrices}

We now investigate when we can choose $p$-lengths on a given curve
to make it into a $p$-obstruction. As a result of the previous
section, if we use $p$-conductances to
search for $p$-obstructions,
we get linear
inequalities and can construct a matrix.

\begin{definition}
  Let $\pi, \phi \co X_1 \rightrightarrows X_0$ be a virtual
  endomorphism,
  let $c \co C \to X_0$ be a multi-curve, and let $1 \le p < \infty$. Then the
  \emph{$p$-obstruction matrix}
  $M_{C,p}$ of~$C$ is the square matrix with rows and columns
  indexed by the components of~$C$, with the $(C_i,C_j)$ entry given
  by
  \[
    \sum_{\substack{D \in \pi^* C_i\\ \phi_* D \sim C_j}}
      \bigl(\deg (D \overset{\pi}\to C_i) \bigr)^{1-p}.
  \]
  In other words, consider all components~$D$ of the multi-curve
  $\pi^* C_i$ on~$X_1$ that push forward to $C_j$, and sum a power of the degree that
  $D$ covers $C_i$. There may be
  components of $\pi^*C_i$ that do not push forward to any $C_j$; these
  components are ignored.
\end{definition}

The matrix $M_{C,p}$ is designed to mimic the action of $\phi_*\pi^*$ on
$p$-conformal multi-curves. Suppose $C$ has $n$ components and let
$\gamma = (\gamma_i)_{i=1}^n \in \RR_{\ge 0}^n$ be a non-negative
vector. Then define $A(\gamma)$ to be the $p$-conformal multi-curve in which
a component~$C_i$ of~$C$ is given $p$\hyp conductance~$\gamma_i$, or is
dropped if $\gamma_i=0$.

\begin{lemma}\label{lem:pullback-matrix}
  For $\gamma$ as above,
  \[
  \phi_*\pi^*A(\gamma) = A(M_{C,p} \gamma) + A'
  \]
  where $A'$ is a $p$-conformal multi-curve whose support is disjoint
  from~$C$.
\end{lemma}

\begin{proof}
  Immediate from Proposition~\ref{prop:join-curves}.
\end{proof}

Observe that $M_{C,p}$ has non-negative
entries. Therefore, by the Perron-Frobenius Theorem, it has an
positive eigenvalue of maximum absolute value with a non-negative
eigenvector. (Our assumptions do not
guarantee that $M_{C,p}$ is irreducible, so the eigenvector might not
be unique.)  Let $\lambda(M_{C,p})$ be this positive eigenvalue.

\begin{proposition}\label{prop:obstruction-e-val}
  Let $\pi,\phi \co X_1 \rightrightarrows X_0$ be a virtual
  endomorphism and let $(C,c)$ be a multi-curve on~$X_0$. Then
  $\lambda(M_{C,p}) \ge 1$ iff there is a
  $p$-obstruction whose support is a sub-multi-curve of~$C$.
\end{proposition}

\begin{proof}
  If $\lambda = \lambda(M_{C,p}) \ge 1$, let $\gamma$ be the
  corresponding non-negative eigenvector. Then by
  Lemma~\ref{lem:pullback-matrix},
  \[
  \phi_*\pi^*A(\gamma) = A(\lambda\gamma) + A'.
  \]
  The $p$-conductances on the components of $A(\gamma)$
  are multiplied by~$\lambda$ in $\phi_*\pi^*A(\gamma)$ and the
  $p$-lengths are multiplied by $\lambda^{-1/(p-1)}$. Thus, tracing
  through the definition of~$E^p_p$ from Equation~\eqref{eq:Epp-std},
  we have
  $
  E^p_p\bigl[A(\gamma) \to \phi_*\pi^*A(\gamma)\bigr] =
  \lambda^{-1/p} \le 1$,
  so $A(\gamma)$ is an obstruction.

  Conversely, suppose that we have a $p$-obstruction~$A$ whose support
  is a sub-multi-curve of~$C$. Form a vector $\gamma$ from the
  $p$-conductances of~$A$, extended by~$0$ for components of~$C$ that
  are not in~$A$. Then Lemma~\ref{lem:pullback-matrix} and the
  assumption that $A$ is a $p$-obstruction say that each component of
  $M_{C,p} \gamma$ is greater than or equal to the corresponding
  component of~$\gamma$.
  By the Collatz-Wielandt formula, this implies that
  $\lambda(M_{C,p}) \ge 1$.
\end{proof}

\begin{corollary}\label{cor:obstruction-iterate}
  If $\pi, \phi \co \Gamma_1 \rightrightarrows \Gamma_0$ has a
  $p$-obstruction, then so do the iterates $\pi^n, \phi^n \co \Gamma_n
  \rightrightarrows \Gamma_0$. In particular, if there is a
  $p$-obstruction for $(\pi,\phi)$, then $\oE^p_p[\pi,\phi] \ge 1$.
\end{corollary}

\begin{proof}
  For any $n \ge 1$, a matrix $M$ has an eigenvalue of absolute value
  greater than~$1$ iff $M^n$ does. Apply
  Propositions \ref{prop:obstruction-e-val} and~\ref{prop:obstruction-energy}.
\end{proof}

Let us investigate what the support of a $p$-obstruction
can look like.

\begin{definition}
  Let $\pi, \phi \co X_1 \rightrightarrows X_0$ be a virtual
  endomorphism and let $(C,c)$ be a multi-curve on~$X_0$. Then $C$ is
  \emph{forwards-invariant} if each component of~$C$ is homotopic to a
  component of $\phi_*\pi^* C$. (Here, $\phi_*$ is defined as in
  Definition~\ref{def:join-top}, but without any $\alpha$-lengths.)
  $C$ is \emph{irreducible} if, for any two
  components $C_i$ and~$C_j$ of~$C$, there is some~$n$ so that $C_j$
  appears in $(\phi_*\pi^*)^n C_i$. (An irreducible curve is
  necessarily forwards-invariant.)
  $C$ is \emph{back-invariant} if each
  component of $\phi_*\pi^* C$ is homotopic to a component of~$C$, and
  $C$ is \emph{totally invariant} if the components
  of~$C$ are in bijection with the components of $\phi_*\pi^* C$.
\end{definition}

The terminology comes from the context of branched self-covers
$f \co (S^2, P)\righttoleftarrow$. Let $C$ be a multi-curve on
$S^2\setminus P$.
\begin{itemize}
\item $C$ is back-invariant iff, up to homotopy in $S^2 \setminus P$,
  we have $C \subset f^{-1}(C)$.
\item $C$ is forwards-invariant iff $C$ is homotopic in
  $S^2 \setminus P$ to a multi-curve $C_1$ on $S^2 \setminus f^{-1}(P)$ with
  $f(C_1) \subset C$.%
  \footnote{Recall that the forward image of a multi-curve in
  $S^2\setminus P$ is not well-defined.}
\end{itemize}
If $A$ is a $p$-obstruction for $(\pi,\phi)$, then the underlying
multi-curve~$C$ of~$A$ must be forwards-invariant. (Otherwise, no
map~$\psi \co A \to \phi_*\pi^* A$ is possible.)
The matrix $M_{C,p}$ is irreducible in the Perron-Frobenius sense iff
$C$ is irreducible as a multi-curve.
On the other hand, in
the context of rational maps it is more traditional to look at
back-invariant multi-curves.

We can often switch between back-invariant and forward-invariant
multi-curves as follows. First, some graph-theory terminology. In a directed
graph, a \emph{strongly-connected component} (SCC) is a maximal set~$S$
of vertices so that every ordered pair of vertices in~$S$ can be
connected by a directed edge-path. Every directed graph is a disjoint
union of its SCCs. A \emph{strict} SCC is an SCC in which every pair
of vertices can be connected by a non-trivial directed edge-path. A
non-strict SCC is a single vertex with no self-loop.

Given a virtual endomorphism $\pi,\phi \co X_1 \rightrightarrows X_0$
and a multi-curve~$C$ on~$X_0$, form the directed graph $\Gamma(C)$
whose vertices are the components of~$C$, with an arrow from $C_i$
to~$C_j$ if $C_j$ appears as a component of $\phi_*\pi^* C_i$. A
strict SCC of $\Gamma(C)$ gives a forward-invariant multi-curve, although
not all forward-invariant multi-curves arise in this way.

\begin{proposition}
  Let $\pi,\phi \co X_1 \rightrightarrows X_0$ be a virtual
  endomorphism and let $C$ be a multi-curve on~$X_0$. Then there is an
  irreducible forwards-invariant
  sub-multi-curve $C_0 \subset C$ with
  $\lambda(M_{C_0,p}) = \lambda(M_{C,p})$.
\end{proposition}

\begin{proof}
  $M_{C,p}$ is block triangular with respect to the partial
  order on the SCCs of $\Gamma(C)$, so its maximum eigenvalue will be
  equal to the
  Perron-Frobenius eigenvalue of a diagonal block corresponding to an
  SCC~$S$. If $S$ is a single vertex with no self-loop, then
  $\lambda(M_{C,p}) = 0$ and we can take $C_0$ to be empty. Otherwise,
  take $C_0$ to be the union of multi-curves in~$S$.
\end{proof}

\begin{proposition}\label{prop:simple-back}
  Let $\pi, \phi \co \Sigma_1 \rightrightarrows \Sigma_0$ be a surface
  virtual endomorphism (with $\phi$ a surface embedding), and let
  $C$ be a simple forward-invariant multi-curve on~$\Sigma_0$. Then
  there is a simple back-invariant multi-curve $C_\infty \supset C$.
\end{proposition}

\begin{proof}
  For $i \ge 1$, let $C_i = \phi_*\pi^* C_{i-1}$ by induction. Then
  each $C_i$ is a simple multi-curve and $C_{i-1} \subset C_i$. Since there
  is a bound on how many components a simple multi-curve on a surface of
  finite type can have, the $C_i$ eventually stabilize into a
  back-invariant multi-curve.
\end{proof}

Although back-invariant multi-curves are more traditional,
forwards-invariant multi-curves appear to be more generally useful.
In a non-surface setting, or if $p < 2$, obstructions need not be
simple and there is no obvious
analogue of Proposition~\ref{prop:simple-back}.

\subsection{Annular obstructions and asymptotic energy}
\label{sec:annular-asymptotic}

Let $\pi, \phi \co \Gamma_1 \rightrightarrows \Gamma_0$ be a
virtual endomorphism.
For any forwards-invariant multi-curve~$C$ on~$\Gamma_0$, there is unique
value of~$p$ so that $\lambda(M_{C,p}) = 1$, which we denote $Q(C)$
\cite[Lemma~A.2]{HP08:Ahlfors-dim}.
Define
$Q(\pi,\phi)$ to be the maximum of $Q(C)$ over all forwards-invariant
multi-curves~$C$.
\begin{proposition}\label{prop:Q-bound}
  Let $\pi,\phi \co \Gamma_1 \rightrightarrows \Gamma_0$ be a graph
  virtual endomorphism. Then if $\oE^p_p[\pi,\phi] \ge 1$, we have
  $Q(\pi,\phi) \le p$.
\end{proposition}
\begin{proof}
  Immediate from Corollary~\ref{cor:obstruction-iterate} and
  Propositions~\ref{prop:obstruction-e-val}
  and~\ref{prop:Epp-decrease}.
\end{proof}
Compare Proposition~\ref{prop:Q-bound} to the following result of
Haïssinsky and Pilgrim.
\begin{citethm}[Haïssinsky--Pilgrim \cite{HP08:Ahlfors-dim}]
  \label{thm:ahlfors-dim}
  Suppose $f \co S^2 \to S^2$ is topologically cxc. Then
  $Q(f) \le \mathop{\mathrm{confdim}}_{AR}(f)$.
\end{citethm}
Here, $Q(f)$ is the version of $Q(\pi,\phi)$ for branched self-covers.
\emph{Topologically cxc} is a
topological notion of expanding
branched self-covers, and in particular implies that
there are no cycles with branch points in~$P$ (the opposite
of the hyperbolic case).
The \emph{Ahlfors
  regular conformal dimension} $\mathop{\mathrm{confdim}}_{AR}(f)$ is
an analytically defined quantity, the minimal Hausdorff dimension of
any Ahlfors regular metric in a certain quasi-symmetry class of
expanding metrics canonically associated to~$f$.

Theorem~\ref{thm:ahlfors-dim}, like Proposition~\ref{prop:Q-bound},
gives upper bounds on
$Q(f)$. However, it gives bounds in terms of the purely analytic
Ahlfors regular conformal dimension, rather than the asymptotic
energy. In addition,
Theorem~\ref{thm:ahlfors-dim}
applies to maps with \emph{no} branched cycles in~$P$, while
Proposition~\ref{prop:Q-bound} is
vacuous unless \emph{every}
cycle in~$P$ is branched. (If there is an unbranched cycle in~$P$,
then $\oE^p_p[\phi] \ge 1$ for every $p \in [1,\infty]$.)

\begin{question}\label{quest:obstruct}
  Suppose $f \co (\Sigma,P) \righttoleftarrow$ is a branched self-cover
  of hyperbolic type, with compatible virtual endomorphism
  $(\pi,\phi)$.  Is it true that $\oE^p_p[\pi,\phi] < 1$
  iff there is no $p$-obstruction?
\end{question}

Theorem~\ref{thm:thurston-obstruction}
and Theorem~\ref{thm:detect-rational} combine to say that the answer to
Question~\ref{quest:obstruct} is positive for $p=2$. However, the
proof is quite roundabout, needing the
full strength of both theorems. One could hope for a more direct proof
of equivalence of the two criteria and a generalization to other
values of~$p$. (This might also give another proof of
Theorem~\ref{thm:thurston-obstruction}.)


\bibliographystyle{hamsalpha}
\bibliography{dylan,conformal,geom,topo,misc,curves}

\end{document}